\documentclass[10pt,reqno]{amsart}
\usepackage[svgnames,table]{xcolor}
\usepackage[colorlinks=true,citecolor=magenta,linkcolor=cyan]{hyperref} 
\usepackage{enumerate}
\usepackage{mathabx}
\usepackage{amssymb,amsfonts, amsmath}
\usepackage{hyperref}
\newtheorem{theorem}{Theorem}[section]
\newtheorem{lemma}[theorem]{Lemma}
\newtheorem{proposition}[theorem]{Proposition}
\newtheorem{corollary}[theorem]{Corollary}
\newtheorem{definition}[theorem]{Definition}

\newtheorem{Example}[theorem]{Example}

\newtheorem{claim}[theorem]{Claim}
\newtheorem{rmk}[theorem]{\normalfont{\em{Remark}}}

\usepackage{environ}
\makeatletter
\newdimen\royalignsep@
\def\royalign@preamble{%
   &\hfil
    \strut@
    \setboxz@h{\@lign$\m@th\displaystyle{##}$}%
    \ifmeasuring@\savefieldlength@\fi
    \set@field
    \tabskip\z@skip
   &\setboxz@h{\@lign$\m@th\displaystyle{{}##}$}%
    \ifmeasuring@\savefieldlength@\fi
    \set@field
    \hfil
    \tabskip\royalignsep@
}
\NewEnviron{royalign}[1]{%
  \royalignsep@=#1\let\align@preamble=\royalign@preamble
  \begin{align}\BODY\end{align}}
\NewEnviron{royalign*}[1]{%
  \royalignsep@=#1\let\align@preamble=\royalign@preamble
  \begin{align*}\BODY\end{align*}}
\makeatother

\usepackage[title]{appendix}

\makeatletter
\renewcommand*\env@matrix[1][\arraystretch]{%
\edef\arraystretch{#1}%
\hskip -\arraycolsep
\let\@ifnextchar\new@ifnextchar
\array{*\c@MaxMatrixCols c}}
\makeatother
\AtEndDocument{\bigskip{\footnotesize%
\textsc{Max Planck Institute for Mathematics in the Sciences, Inselstra{\ss}e 22, 04103 Leipzig, Germany } \par  
 \textit{E-mail address}: \texttt{konstantinos.tsouvalas@mis.mpg.de} }}
\date{\today}
\title[Robust quasi-isometric embeddings of virtually free groups]{Robust quasi-isometric embeddings of virtually free groups}
\date{\today}

\author{Konstantinos Tsouvalas}

\begin{document}

\frenchspacing

\maketitle
\begin{abstract} Let $k$ be a nonarchimedean local field. For any $n\geq 3$, we construct the first examples of robust quasi-isometric embeddings of non-elementary free groups into $\mathsf{GL}_n(k)$ which are not limits of Anosov representations. If $\bf{K}=\mathbb{R},\mathbb{C}$, we exhibit examples of non-locally rigid, robust quasi-isometric embeddings of virtually free groups into $\mathsf{GL}_n(\bf{K})$, $n\geq 3$, which are not limits of Anosov representations. Moreover, we exhibit a non-Anosov robust quasi-isometric embedding of the free semigroup $\mathbb{Z}\ast \mathbb{Z}^{+}$ into $\mathsf{GL}_3(\mathbb{C})$, which is a limit of Anosov representations.\end{abstract}
\medskip

\section{Introduction}

Let $k$ be a local field. A representation $\rho:\Gamma \rightarrow \mathsf{GL}_n(k)$, $n\geq 2$, of a finitely generated group $\Gamma$ is called a {\em quasi-isometric embedding} if the orbit map of $\rho$ is a quasi-isometric embedding of $\Gamma$ into the Riemannian symmetric space (resp. Bruhat--Tits building) of $\mathsf{GL}_n(k)$, where $k$ is Archimedean (resp. nonarchimedean). The representation $\rho$ is called a {\em robust quasi-isometric embedding}\footnote{Throughout this paper, a robust quasi-isometric embedding of $\Gamma$ into $\mathsf{GL}_n(k)$ will always refer to a representation of $\Gamma$ into $\mathsf{GL}_n(k)$.} (resp. {\em robustly discrete-faithful}) if there exists an open neighborhood of $\rho$, in the space of representations $\textup{Hom}(\Gamma, \mathsf{GL}_n(k))$, containing entirely quasi-isometric embeddings (resp. discrete-faithful representations).

Anosov representations of Gromov hyperbolic groups comprise a rich class of robust quasi-isometric embeddings into reductive groups. Anosov representations into real reductive Lie groups were introduced by Labourie \cite{Labourie} and further developed by Guichard--Wienhard \cite{GW}; the definition extends for representations of hyperbolic groups into the $k$-points of reductive groups defined over a nonarchimedean local field $k$, see \cite[Def. 5.31]{KLP2}, \cite[Def. 3.1]{PSW21}, and \cite[Rmk. 1.6]{GGKW}. Sullivan in \cite{Sullivan} established that a robustly faithful\footnote{i.e., a faithful representation all of whose nearby deformations are faithful.} non-locally rigid representation of a finitely generated group into $\mathsf{PSL}_2(\mathbb{C})$ is convex cocompact and hence Anosov. Motivated by Sullivan's structural stability theorem, R. Potrie and F. Kassel independently asked in \cite[Q.5]{Potrie-ICM} and \cite[Q.1]{Kassel-ICM} whether a non-locally rigid robust quasi-isometric embedding (resp. robustly discrete-faithful representation) of a hyperbolic group into a higher rank real semisimple Lie group $\mathsf{G}$ is Anosov. A negative answer to the previous questions is known when $\mathsf{G}=\mathsf{SL}_5^{\pm}(\mathbb{R})$ \cite[Ex. 8.5]{DGKLM21}, $\mathsf{G}=\mathsf{SL}_n(\mathbb{C})$, $n\geq 30$ \cite[Cor. 1.2]{Tso-robust} and $\mathsf{G}=\mathsf{SL}_n(\mathbb{R})$ for $n\geq 756$ \cite[Thm. 1.1]{Tso-robust} (see also Remark \ref{rmk-378}). A positive answer to \cite[Q.5]{Potrie-ICM} was given in \cite{CPL24} when $\mathsf{G}=\mathsf{SL}_2(\mathbb{R})^r\times \mathsf{SL}_2(\mathbb{C})^s$, $r+s\geq 1$ and $\mathsf{G}=\mathsf{SL}_3(\mathbb{R})$ for reducible representations of non-elementary free groups.

%In particular, for any $n\geq 30$ there exist robust quasi-isometric embeddings of torsion-free hyperbolic groups into $\mathsf{SL}_n(\mathbb{C})$ that fail to be limits of Anosov representations, see \cite[Cor. 1.2]{Tso-robust}. For representations of free groups into 

In the present paper, we construct examples of robust quasi-isometric embeddings of virtually free groups which fail to be limits of Anosov representations. The first examples of Zariski dense, quasi-isometric embeddings of (non-elementary) free groups into real semisimple Lie groups, which are not limits of Anosov representations, were constructed by Carvajales--Lessa--Potrie in \cite[Thm. 1.5]{CPL24}. Such examples are limits of representations of free groups whose images contain unipotent elements, and hence they are not robust quasi-isometric embeddings. Other examples of strongly irreducible (but not Zariski dense) quasi-isometric embeddings of surface groups into $\mathsf{SL}_m(\mathbb{R})$, $m\in \{4,6\}$, which fail to be limits of Anosov representations into $\mathsf{SL}_m(\mathbb{R})$, were exhibited in \cite[Prop. 4.2]{Tso-limit}. However, such examples are Anosov when considered as subgroups of their Zariski closure, and it is also unlikely that they are robustly discrete into $\mathsf{SL}_m(\mathbb{R})$. At this point, it is an open problem whether a robust quasi-isometric embedding of a (non-elementary) free group into a higher rank, connected, real semisimple Lie group has to be Anosov with respect to a pair of opposite parabolic subgroups of the target group; see \cite[Prob. 1.1]{CPL24}.

For a semigroup $\Gamma$ with identity and $ j\in \{1,\ldots, n-1\}$, denote by $\textup{Anosov}_j(\Gamma,\mathsf{GL}_n(k))$ the subset of $j$-Anosov representations of $\Gamma$ into $\mathsf{GL}_n(k)$  (see also Definition \ref{Anosov-def}). A representation of $\Gamma$ into $\mathsf{GL}_n(k)$ \hbox{is {\em Anosov} if it is $j$-Anosov for some $j$.} The first result of this paper provides a negative answer to the nonarchimedean analogue of \cite[Prob. 1.1]{CPL24}.

\begin{theorem}\label{nonArchimedean} Let $k$ be a nonarchimedean local field and integers $n\geq 3$, $m\geq 2$. There exists a robust quasi-isometric embedding of the free group of rank $m$ into $\mathsf{GL}_n(k)$ which is not a limit of Anosov representations into $\mathsf{GL}_n(k)$.\end{theorem}
Theorem \ref{nonArchimedean} shows that nonarchimedean analogues of Sullivan's stability theorem \cite{Sullivan} and of the density theorem of Kleinian groups \cite{Brbr, BCM, NS, Ohshika} fail for Anosov representations of free groups into $\mathsf{GL}_n(k)$, $n\geq 3$. To our knowledge, these are the first examples of robustly discrete free subgroups of reductive $k$-groups that fail to be Anosov. The examples of Theorem \ref{nonArchimedean} are constructed as the ping-pong combination of a non-proximal\footnote{We refer to a diagonal matrix with at least two eigenvalues of maximum possible absolute value.} diagonal matrix  with an Anosov free subgroup of rank $m-1$, following the point of view of the construction in \cite[Thm. 1.5]{CPL24} in dimension $3$. For the proof of robustness of the examples of Theorem \ref{nonArchimedean} we crucially use the ultrametric property of the nonarchimedean absolute value on $k$. 

%We refer the reader to Example \ref{Exmp1} for an explicit example.

 In addition, we also exhibit examples of robust quasi-isometric embeddings of virtually free groups into real and complex general linear groups which fail to be Anosov. 

\begin{theorem}\label{Archimedean}  Let $d\geq 2$ be an even integer. There exists a virtually free group $\Gamma$, depending only on $d$, with the property: for every $n\geq d+1$, there is a representation $\psi_n:\Gamma\rightarrow \mathsf{GL}_n(\mathbb{R})$ and an open neighborhood $\mathcal{O}_n$ of $\psi_n$ in $\textup{Hom}(\Gamma ,\mathsf{GL}_n(\mathbb{C}))$ satisfying the following: \begin{enumerate}
 \item\label{Arch-det-1}  Every representation in $\mathcal{O}_{n}$ is a quasi-isometric embedding and non-locally rigid;
\item \label{Arch-det-2}  $\mathcal{O}_{n} \cap \textup{Anosov}_j\left(\Gamma,\mathsf{GL}_n(\mathbb{C})\right)$ is empty for any $j=1,\ldots,d-1$; and 
\item \label{Arch-det-3} $\mathcal{O}_{n} \cap \textup{Anosov}_r\left(\Gamma,\mathsf{GL}_n(\mathbb{R})\right)$ is empty for any $r=1,\ldots,n-1$.\end{enumerate}\end{theorem}

%In particular, when $n=d+1$ or $n=d+2$, Theorem \ref{Archimedean} provides an example of a non-locally rigid robust quasi-isometric embedding of a virtually free group $\Gamma$ into $\mathsf{GL}_n(\mathbb{C})$, with the property that $\textup{Anosov}_{r}(\Gamma, \mathsf{GL}_n(\mathbb{C}))$ is empty for any $r$.

If $n\in \{d+1,\ldots, 2d-1\}$ the representation $\psi_n$ in Theorem \ref{Archimedean} is not in the closure of Anosov representations of $\Gamma$ into $\mathsf{GL}_n(\mathbb{C})$. The representation $\psi_n$, as well as any representation in the open set $\mathcal{O}_n$, from Theorem \ref{Archimedean} is obtained as a ping-pong combination of a virtually cyclic group, with non-trivial torsion, with infinite cyclic groups. As a consequence, this provides more flexible (in terms of the dimension of the space of deformations) examples of non-Anosov robust quasi-isometric embeddings into any dimension $\geq 3$. While the proof of the robustness of the examples in \cite{Tso-robust} relies crucially on Corlette's Archimedean superrigidity \cite{Corlette}, the proof of Theorem \ref{Archimedean} uses the local rigidity of the finite subgroups of $\Gamma$. We refer  to Example \ref{Exmp2} for an explicit non-Anosov robust quasi-isometric embedding of the free product $(\mathbb{Z}\times \mathsf{D}_8)\ast \mathbb{Z}$ into $\mathsf{GL}_3(\mathbb{C})$, where $\mathsf{D}_8$ is the dihedral group of order $8$.

%provide examples of non-Anosov robst quasi-isometric embeddings into $\mathsf{SL}_n({\bf K})$, when ${\bf K}=\mathbb{R}$ or $\mathbb{C}$ and $n\geq 30$, 

%\begin{theorem}\label{Archimedean} Let $F$ be a finite group with an irreducible, faithful representation of even dimension \hbox{$d$} over $\mathbb{R}$ and for $m\geq 1$ consider the virtually free group $$\Gamma_{m}:= \big(\mathbb{Z}\times F\big)\ast \mathsf{F}_{m}.$$ For every $n\geq d+1$ there exists a representation \hbox{$\psi_n: \Gamma_{m} \rightarrow \mathsf{GL}_n(\mathbb{R})$}, which is not Anosov, and an open neighborhood of $\mathcal{O}_n$ of $\psi_n$ in $\textup{Hom}(\Gamma_{m} ,\mathsf{GL}_n(\mathbb{C}))$ with the following properties:\begin{enumerate}
% \item Every representation in $\mathcal{O}_{n}$ is a quasi-isometric embedding;
%\item  Every representation in $\mathcal{O}_{n}$ is not $j$-Anosov for any $j=1,\ldots,d-1$; and 
%\item $\mathcal{O}_{n} \cap \textup{Anosov}_r\left(\Gamma_{m},\mathsf{GL}_n(\mathbb{R})\right)$ is empty for every $r=1,\ldots,n-1$.\end{enumerate}\end{theorem}

Anosov representations of semigroups into a real semisimple Lie group were introduced by Kassel and Potrie in \cite[\S 5]{KP}. We exhibit examples of robust quasi-isometric embeddings of free semigroups which are not Anosov, thus providing a negative answer to an analogue of \cite[Prob. 1.1]{CPL24} for free semigroups. We denote by $\mathbb{Z}^{+}$ the semigroup of natural numbers equipped with multiplication.

\begin{theorem}\label{Archimedean-sem}There exist $A,B\in \mathsf{SL}_3(\mathbb{R})$ diagonalizable matrices, and open neighborhoods $\Omega_1$ and $\Omega_2$ of $A$ and $B$ in $\mathsf{GL}_3(\mathbb{C})$ respectively, satisfying the following:\begin{enumerate}
\item \label{Arch-sem1} The semigroup $\langle A,A^{-1},B\rangle $ is isomorphic to $\mathbb{Z}\ast \mathbb{Z}^{+}$ and not Anosov into $\mathsf{GL}_3(\mathbb{C})$;
 \item \label{Arch-sem2} For any $A'\in \Omega_1$ and $B'\in \Omega_2$, the semigroup $\langle A',(A')^{-1},B'\rangle$ is quasi-isometrically embedded into $\mathsf{GL}_3(\mathbb{C})$ and isomorphic to $\mathbb{Z}\ast \mathbb{Z}^{+}$; and
\item \label{Arch-sem3} For any $A'\in \Omega_1$ and $B'\in \Omega_2$, the semigroup $\langle A',(A')^{-1},B'\rangle$  is Anosov into $\mathsf{GL}_3(\mathbb{C})$ if and only if the moduli of the eigenvalues of $A'$ are distinct. \end{enumerate}\end{theorem}

\subsection*{Organization of the paper} In Section \ref{background} we provide some preliminary lemmas that we need for the proof of the main theorems. In Sections \ref{nonArch}, \ref{Arch} and \ref{Sem-C} we prove Theorems \ref{nonArchimedean}, \ref{Archimedean} and \ref{Archimedean-sem} respectively. In Section \ref{Examples} we exhibit two explicit examples of non-Anosov robust quasi-isometric embeddings in dimension 3.

\subsection*{Acknowledgements.} The author would like to thank  Le\'on Carvajales, Sami Douba, Pablo Lessa, Rafael Potrie, Teddy Weisman and Anna Wienhard for interesting discussions. The author also acknowledges support from the Simons Laufer Mathematical Sciences Institute in Berkeley, California, during the Spring 2026 semester (National Science Foundation Grant No. DMS2424139).

\section{Preliminaries} \label{background} %For $m\in \{1,\ldots,n\}$ we set $e_m^{\perp}:=\textup{span}\{e_j: j  \neq m\}$. 

Throughout this paper $k$ denotes a local field. We recall that $k$ is Archimedean if it is isomorphic to $\mathbb{R}$ or $\mathbb{C}$, otherwise $k$ is nonarchimedean and isomorphic to a finite extension of the field of $p$-adic numbers $\mathbb{Q}_p$ or the field of formal Laurent series $\mathbb{F}_q((T))$ over the finite field $\mathbb{F}_q$. 
Denote by $|\cdot|:k\rightarrow \mathbb{R}_{+}$ the absolute value on $k$. If $k$ is Archimedean, $|\cdot|$ is the Euclidean absolute value; if $k$ is nonarchimedean, $|\cdot|$ is the absolute value induced by the discrete valuation on $k$ satisfying the ultrametric property \begin{align}\label{ultrametric}  |x_1+\cdots+x_r|\leq \max \big\{|x_1|,\ldots, |x_r|\big\}, \ \forall x_1,\ldots,x_r \in k. \end{align}

The $k$-vector space $k^n$, $n\geq 2$, is equipped with the canonical basis $(e_1,\ldots,e_n)$. If $k$ is Archimedean, $k^n$ is equipped with the standard Euclidean norm $||\cdot||$; otherwise, equip $k^n$ with the $\ell_{\infty}$-norm $||\cdot||$, where $\big|\big|a_1e_1+\cdots+a_ne_n\big|\big|=\max_{j}|a_j|$.
Equip the projective space $\mathbb{P}(k^n)$ with the metric $d_{\mathbb{P}}$ defined as follows $$d_{\mathbb{P}}\big([u],[v]\big)=\min\big\{ | |\zeta u-v| |:|\zeta|=1\big\}, \ ||u||=||v||=1.$$ For $x\in \mathbb{P}(k^n)$ and  $\epsilon>0$, $B_{\epsilon}(x)\subset \mathbb{P}(k^n)$ is the closed ball of radius $\epsilon>0$ centered at $x$. For $u=(u_1,\ldots,u_n)^t\in k^n$, $v=(v_1,\ldots,v_n)^t\in k^n$ define their dot product as $$u\cdot v:=u_1\overline{v}_1+\cdots+u_n\overline{v}_n$$ where, for $x\in k$, if $k$ is Archimedean, \hbox{$\overline{x}=x$ is the complex conjugate of $x$, otherwise $\overline{x}=x$.}

For $v\in k^n$ non-zero consider the $(n-1)$-hyperplane $v^{\perp}:=\big\{u\in k^n: u\cdot v=0\}$. By the definition of $d_{\mathbb{P}}$, for any non-zero vector $\omega \in k^n$ the inequality folds \begin{align}\label{distance}\textup{dist}\big([\omega],\mathbb{P}(v^{\perp})\big)\geq \frac{ |\omega \cdot v|}{||\omega||\cdot ||v||}.\end{align} %where $\textup{dist}\big([\omega],\mathbb{P}(v^{\perp})\big)=\min \big\{d_{\mathbb{P}}([\omega],[\omega']):[\omega']\in \mathbb{P}(v^{\perp})\big\}$.

\subsection{Quasi-isometric embeddings and Anosov representations} For background on reductive groups over local fields, symmetric spaces and Bruhat--Tits buildings, see \cite{Helgason, BT72}. For a matrix $g\in \mathsf{GL}_n(k)$, denote by $\sigma_1(g)\geq \cdots \geq \sigma_n(g)$ (resp. $\ell_1(g)\geq \cdots \geq \ell_d(g)$) the singular values (resp. moduli of eigenvalues) of $g$ in non-increasing order. For $1\leq i\leq n-1$, the equality $\ell_i(g)=\underset{m\rightarrow \infty}{\lim}\sqrt[m]{\sigma_i( g^m)}$ \hbox{holds, see \cite[\S 2.5]{Ben97}.}

Let $\Gamma$ be a finitely generated semigroup with identity, equipped with the word metric from a finite generating subset and let $|\cdot|_{\Gamma}:\Gamma\rightarrow \mathbb{R}_{+}$ be the function mapping $\gamma\in \Gamma$ to its distance from the identity. A representation $\rho:\Gamma \rightarrow \mathsf{GL}_n(k)$ is called a {\em quasi-isometric embedding} \hbox{if there exist $c,\theta>0$ such that} $$\forall  \gamma \in \Gamma, \ \frac{\sigma_1(\rho(\gamma))}{\sigma_n(\rho(\gamma))}\geq ce^{\theta |\gamma|_{\Gamma}}.$$

Anosov representations were introduced in \cite{Labourie} for fundamental groups of closed negatively curved Riemannian manifolds and extended in \cite{GW} for general hyperbolic groups. For the definition of Anosov representations over nonarchimedean local fields, see \cite[Def. 5.31]{KLP2}, \cite[Def. 3.1]{PSW21}, and also \cite[Rmk. 1.6]{GGKW}. Anosov representations of semigroups were introduced in \cite[\S 5]{KP}. For more background on the theory of Anosov representations, \hbox{see \cite{GGKW, KLP1, Kassel-ICM, KP, Wienhard-ICM, Canary-ICM}.}

\begin{definition}\label{Anosov-def} A representation $\rho:\Gamma \rightarrow \mathsf{GL}_n(k)$ is $j$-Anosov, $j\in \{1,\ldots,n-1\}$, if there exist $c, \varepsilon>0$ such that $$ \forall  \gamma \in \Gamma, \ \frac{\sigma_{j}(\rho(\gamma))}{\sigma_{j+1}(\rho(\gamma))}\geq ce^{\varepsilon |\gamma|_{\Gamma}}.$$ \end{definition}

In the case where $\Gamma$ is a group, the previous definition follows by a characterization of Kapovich--Leeb--Porti \cite{KLP2} and Bochi--Potrie--Sambarino \cite{BPS}. A representation $\rho:\Gamma \rightarrow \mathsf{GL}_n(k)$ is called {\em Anosov} if it is $j$-Anosov for some $j$. In addition, if $\Gamma$ is a group then it has to be Gromov-hyperbolic \cite{BPS,KLP2}. 

\subsection{Proximality} For $1\leq j \leq n-1$, $g\in \mathsf{GL}_d(k)$ is called {\em j-proximal} if  $\ell_j(g)>\ell_{j+1}(g)$. The matrix $g$ is called $j$-biproximal if both $g$ and $g^{-1}$ are $j$-proximal.
 If $g$ is $1$-proximal, $g$ has a unique attracting fixed point $g^{+}\in \mathbb{P}(k^n)$ and a unique repelling hyperplane $V_{g}^{+}\in \mathsf{Gr}_{n-1}(k^n)$ such that $\underset{r\rightarrow \infty}{\lim} g^rx=g^{+}$ for every $x\in \mathbb{P}(k^n)\smallsetminus \mathbb{P}(V_g^{-})$.

If $g\in \mathsf{GL}_n(k)$ is $1$-biproximal, set $g^{-}:=(g^{-1})^{+}$, $V_{g^{-1}}^{-}:=V_{g}^{+}$ and for $\theta>0$ let \begin{align*} \mathcal{M}_{\theta}(g)&:=\mathbb{P}(k^n)\smallsetminus \mathcal{N}_{\theta}(\mathbb{P}(V_g^{+}))\cup \mathcal{N}_{\theta}(\mathbb{P}(V_g^{-})),\\ \mathcal{M}_{\theta}^{\pm}(g)&:=\mathbb{P}(k^n)\smallsetminus \mathcal{N}_{\theta}(\mathbb{P}(V_g^{\pm })). \end{align*} 

For $h\in \mathsf{GL}_n(k)$, set $C(h):=2||h||\cdot||h^{-1}||$; the action of $h$ on $(\mathbb{P}(k^n),d_{\mathbb{P}})$ is $C(h)$-Lipschitz (see Lemma \ref{dist} (ii)). We will need the following elementary estimate on the Lipschitz constant for the action of a proximal element $g\in \mathsf{GL}_d(k)$ on the set $\mathcal{M}_{\theta}(g)$. 

\begin{lemma}\label{Lip} Let $h\in \mathsf{GL}_n(k)$, $n\geq 2$, and $g\in \mathsf{SL}_n(k)$ a 1-biproximal matrix of the form \begin{align}\label{proximal-g}g=h\begin{pmatrix} \kappa_1 & &\\ & A_g & \\ & & \kappa_n\end{pmatrix}h^{-1},\   \begin{split}
 |\kappa_1|&>\max\big\{||A_g||, \ell_1(A_g)\big\} \\ |\kappa_n^{-1}|&>\max\big\{||A_g^{-1}||, \ell_1(A_g^{-1})\big\}
\\ A_g & \in \mathsf{GL}_{n-2}(k).\end{split} \end{align} For $\theta>0$ and $p\in \mathbb{N}$, the restriction of $g^{\pm p}$ on $\big(\mathcal{M}_{\theta}^{\mp}(g),d_{\mathbb{P}}\big)$ is $\mathcal{L}(g,\theta)$-Lipschitz, where $\mathcal{L}(g, \theta):=\frac{4C(h)^4}{\theta^{2}}\max\big\{ |\kappa_1^{-1}|\cdot ||A_g||, |\kappa_n|\cdot ||A_g^{-1}||\big\}^{p}.$ In addition, for any $[v]\in \mathcal{M}_{\theta}^{\mp}(g)$, \begin{align}\label{norm-power} \big| \big|g^p v\big|\big| \geq \frac{2\theta}{ C(h)^2}\min \big\{|\kappa_1|,|\kappa_n^{-1}|\big\}^{|p|}| | v | |.\end{align}\end{lemma}

We postpone the proof of Lemma \ref{Lip} in the appendix. By using the previous estimates we prove the following lemma for perturbations of biproximal matrices.

\begin{lemma}\label{perturb1} Let $h\in \mathsf{GL}_n(k)$ and $g\in \mathsf{GL}_n(k)$ be a 1-biproximal matrix \hbox{as in (\ref{proximal-g}). Fix $\alpha, \delta,\varepsilon>0$ with} $$0<\varepsilon\leq \min\big\{\textup{dist}(g^{+},\mathbb{P}(V_{g}^{-})),\textup{dist}(g^{-},\mathbb{P}(V_{g}^{+}))\big\}, \ 0<\delta\leq 10^{-2}\varepsilon$$ and let $U_{\delta}:=\big\{w\in \mathsf{GL}_n(k): | |w^{\pm 1}-\textup{I}_n| |<\frac{\delta}{10}\big\}.$ Suppose that $\kappa_1,\kappa_n\in k$ and $A_g\in \mathsf{GL}_{n-2}(k)$ from (\ref{proximal-g}) satisfy the following conditions \begin{align}\label{condition1} \min \left\{\frac{|\kappa_1|}{||A_g||}, \big(|\kappa_n|\cdot ||A_g^{-1}||\big)^{-1}\right\}&\geq 10^6 e^{\alpha+8}(\varepsilon^{2}\delta)^{-1} C(h)^4 \end{align} \begin{align} \label{condition2} \min \Big\{|\kappa_1|, |\kappa_n^{-1}| \Big\}&\geq 10^3 e^{\alpha+8}\varepsilon^{-1} C(h)^2. \end{align} For any $g'\in gU_{\delta}$, $p\neq 0$ and $[v]\in \mathcal{M}_{\varepsilon}(g)$, \begin{align}(g')^p\mathcal{M}_{\varepsilon}(g)&\subset B_{\delta}(g^{+})\cup B_{\delta}(g^{-})\\ \big|\big| (g')^pv\big|\big|&\geq e^{\alpha |p|} | |v||.\end{align} \end{lemma}

\begin{proof} For $w\in U_{\delta}$ note that $||w^{\pm 1}||\leq 2$ and $C(w)=2||w||\cdot ||w^{-1}||\leq 8$. By Lemma \ref{dist}, as $0<\delta\leq 10^{-2}\varepsilon$, for any $w\in U_{\delta}$ we have the inclusions \begin{align}\label{dist2} w^{\pm 1}B_{\delta}(g^{\pm})\subset B_{9\delta}(g^{\pm}), \ w^{\pm 1} \mathcal{M}_{\varepsilon}^{\pm}(g)\subset \mathcal{M}_{\frac{\varepsilon}{16}}^{\pm}(g), \ w^{\pm 1} \mathcal{M}_{\frac{\varepsilon}{16}}^{\pm}(g)\subset \mathcal{M}_{\frac{\varepsilon}{10^3}}^{\pm}(g).\end{align}

By (\ref{condition1}) and Lemma \ref{Lip}, $\mathcal{L}(g,10^{-3}\varepsilon)\leq \delta e^{-(\alpha+6)}$ and for any $r\in \mathbb{N}$ the restriction of $g^{\pm r}$ on $\mathcal{M}_{{\frac{\varepsilon}{10^3}}}^{\mp}(g)$ is $\delta e^{-(\alpha+6)r}$-Lipschitz. In addition, as $B_{14\delta}(g^{\pm})\subset \mathcal{M}_{\frac{\varepsilon}{10^3}}^{\mp}(g)$, by (\ref{condition2}), \hbox{for every $r\in \mathbb{N}$,} \begin{align} \label{inclusion1}g^{\pm r}B_{\theta}(g^{\pm})&\subset B_{\theta e^{-(\alpha+6) r}}(g^{\pm}), \ 0<\theta<14\delta\\ \label{inclusion2} g^{\pm r}\mathcal{M}_{\frac{\varepsilon}{10^3}}(g)&\subset B_{\delta e^{-(\alpha+6) r}}(g^{\pm})\\ \label{estimate3} \ \forall  [v]\in \mathcal{M}_{\frac{\varepsilon}{10^3}}^{\mp}&(g), \  | | g^{\pm r} v| | \geq 4e^{\alpha r} | | v| | .\end{align} 

 Now we check that every $g'\in U_{\delta}gU_{\delta}$,  where $U_{\delta}gU_{\delta}=\{h_1gh_2:h_1,h_2\in U_{\delta}\}$, satisfies the conclusion of the lemma. For any $h_1,h_2 \in U_{\delta}$, by (\ref{dist2}), (\ref{inclusion1}) and (\ref{inclusion2}) we have the inclusions \begin{align*} h_1g^{\pm 1} h_2 \mathcal{M}_{\varepsilon}(g)&\subset h_1 g^{\pm 1} \mathcal{M}_{\frac{\varepsilon}{16}}(g)\subset h_1B_{\frac{\delta}{e^{6}}}(g^{\pm})\subset B_{\frac{8\delta}{e^{6}}+\frac{2\delta}{5}}(g^{\pm })\subset B_{\delta}(g^{\pm})\\ h_1g^{\pm 1}h_2B_{\delta}(g^{\pm})&\subset h_1g^{\pm 1}B_{9\delta}(g^{\pm}) \subset h_1B_{\frac{9\delta}{e^{6}}}(g^{\pm}) \subset h_1B_{\frac{\delta}{40}}(g^{\pm})\subset B_{\delta}(g^{\pm})\end{align*} and thus for every $p\in \mathbb{N}$, $(h_1g^{\pm 1}h_2)^p\mathcal{M}_{\varepsilon}(g)\subset B_{\delta}(g^{\pm}).$ Note also for $p\in \mathbb{N}$, $h_1,h_2\in U_{\delta}$  and $[v]\in \mathcal{M}_{\varepsilon}(g)$, by (\ref{dist2}) and the previous inclusion, $$[h_2(h_1g^{\pm 1}h_2)^{p-1}v] \in h_2\mathcal{M}_{\frac{\varepsilon}{16}}(g)\cup h_2B_{\delta}(g^{\pm })\subset h_2\mathcal{M}_{\frac{\varepsilon}{16}}(g)\cup B_{13\delta}(g^{\pm })\subset \mathcal{M}_{{\frac{\varepsilon}{10^3}}}^{\mp}(g).$$ By using (\ref{estimate3}) we obtain the bounds \begin{align*}\big|\big|(h_1g^{\pm 1}h_2)^pv\big|\big|& \geq ||h_1^{-1}||^{-1}  \big|\big|g^{\pm 1}(h_2(h_1g^{\pm 1}h_2)^{p-1}v)\big|\big|\\
&\geq 4e^{\alpha} ||h_1^{-1}||^{-1}  \big|\big|h_2(h_1g^{\pm 1}h_2)^{p-1}v\big|\big|\\ &\geq 4e^{\alpha} (||h_2^{-1}||\cdot ||h_1^{-1}||)^{-1} \big|\big|(h_1g^{\pm 1}h_2)^{p-1}v\big|\big|\\ &\geq e^{\alpha}\big|\big|(h_1g^{\pm 1}h_2)^{p-1}v\big|\big|.\end{align*} This shows that for every $p\in \mathbb{N}$ and $h_1,h_2\in U_{\delta}$, $|| (h_1g^{\pm 1}h_2)^p v|| \geq e^{\alpha p} | |v | |.$  
Thus, the open neighbourhood $gU_{\delta}$ of $g$ in $\mathsf{GL}_n(k)$ satisfies the conclusion of the lemma.\end{proof}

From now own, $\mathsf{F}_m$ is the free group on the finite set $\{a_1^{\pm 1},\ldots,a_m^{\pm 1}\}$ equipped with the induced word metric. For $g\in \mathsf{F}_m$, $|g|$ is the word length of $g$.

\begin{lemma}\label{perturb2} Let $(x,V)\in \mathcal{F}_{1,n-1}(k)$. Fix $m\in \mathbb{N}$, $\alpha>0$ and $0<\varepsilon<1$. There exists a representation $\psi:\mathsf{F}_m\rightarrow \mathsf{SL}_n(k)$ and an open neighborhood $\Omega\subset \textup{Hom}(\mathsf{F}_m,\mathsf{GL}_n(k))$ with the property that for every $g\in \mathsf{F}_m$, $g\neq 1$ and $\psi'\in \Omega$, $$\psi'(g)\big(\mathbb{P}(k^n)\smallsetminus \mathcal{N}_{\varepsilon}(\mathbb{P}(V))\big)\subset B_{\varepsilon}(x)$$ and for every $[v]\in \mathbb{P}(k^n)\smallsetminus \mathcal{N}_{\varepsilon}(\mathbb{P}(V))$,  $$\big| \big| \psi'(g)v \big| \big|\geq e^{\alpha |g|} | |v | |.$$ Moreover, if $n\geq 4$ and $k$ is either nonarchimedean or $\mathbb{R}$, we may choose $\psi$ and $\Omega$ such that any representation $\psi'\in \Omega$ is not $2q$-Anosov for any $1\leq q\leq \frac{n-1}{2}$. \end{lemma}

\begin{proof} Choose $0<\epsilon\leq \varepsilon$ and pairwise transverse flags $(x_1,V_1),\ldots, (x_{2m},V_{2m})\in \mathcal{F}_{1,n-1}(k)$ such that $B_{\epsilon}(x_{s})\subset B_{\varepsilon}(x)$, $\mathcal{N}_{\epsilon}(\mathbb{P}(V_{\ell}))\subset \mathcal{N}_{\varepsilon}(\mathbb{P}(V))$  for every $s=1,\ldots, m$ and $$B_{\epsilon}(x_r) \subset \mathbb{P}(k^n)\smallsetminus \mathcal{N}_{\epsilon}(\mathbb{P}(V_{s})), \ \ s\neq r.$$  By Lemma \ref{perturb1}, there are $g_1,\ldots,g_m\in \mathsf{SL}_n(k)$ $1$-biproximal matrices with $(g_i^{+},V_{g_i}^{+})=(x_{2i-1},V_{2i-1})$, $(g_i^{-},V_{g_i}^{-})=(x_{2i},V_{2i})$ for every $i$ and $B_{\epsilon}(g_i^{\pm})\subset \mathcal{M}_{\epsilon}(g_j)$ for $i\neq j$, and open neighbourhoods $\Omega_{i}\subset \mathsf{GL}_d(k)$ of $g_i$  such that for every $g_i'\in \Omega_i$    and $p\neq 0$,  \begin{align}\label{estimate0} (g_i')^p \mathcal{M}_{\epsilon}(g_i)&\subset B_{\epsilon}(g_i^{+})\cup B_{\epsilon}(g_i^{-}),\\ \label{estimate0'}  \forall  [v]\in \mathcal{M}_{\epsilon}(g_i),& \ | |(g_i')^pv| |\geq e^{\alpha|p|}|| v||.\end{align}

The unique representation $\psi:\mathsf{F}_m\rightarrow \mathsf{SL}_n(k)$ with $\psi(a_i)=g_i$ for every $i$, and its open neighborhood $\Omega:= \Omega_{1}\times \cdots \times \Omega_{m}$ satisfy the conclusion of the lemma.
To see this, fix a reduced word  $a_{i_r}^{s_r}\cdots a_{i_1}^{s_1}\in \mathsf{F}_m$, where $r\geq 2$ and $s_j\neq 0$, and a representation $\psi'\in \Omega$ with $\psi'(a_i)=g_i'$. For any $[v]\in \mathbb{P}(k^n)\smallsetminus \mathcal{N}_{\varepsilon}(\mathbb{P}(V))$, as $[v]\in \mathcal{M}_{\epsilon}(g_{i_r})$ and $B_{\epsilon}(g_{i_r}^{\pm})\subset \mathcal{M}_{\epsilon}(g_{i_s})$ when $r\neq s$, we inductively check that for every $p=1,\ldots, r-1$, \begin{align*} \big[(g_{i_{p}}')^{s_{p}}\cdots (g_{i_1}')^{s_1}v\big]&\in B_{\epsilon}(g_{i_{p}}^{\pm})\subset \mathcal{M}_{\varepsilon}(g_{i_{p+1}}),\\ \big|\big|(g_{i_{p}}')^{s_p}\cdots (g_{i_1}')^{s_1}v\big|\big|&\geq e^{\alpha \sum |s_i|}||v||.\end{align*} Therefore, $[\psi'(a_{i_r}^{s_r}\cdots a_{i_1}^{s_1})v]\in B_{\epsilon}(g_{i_{r}}^{\pm})\subset B_{\delta}(x)$ and $\big|\big |\psi'(a_{i_r}^{s_r}\cdots a_{i_1}^{s_1})v\big|\big|\geq e^{\alpha\sum |s_i|} | |v | |$.
\medskip

\par For the rest of the proof  we assume that $n\geq 4$ and  choose $w\in \mathsf{GL}_n(k)$ such that $(x_1,V_1)=\big([we_1],we_n^{\perp}\big)$ and $(x_2,V_2)=\big([we_n], w e_1^{\perp}\big)$ are the attracting and repelling $\{1,n-1\}$ flags of $\psi(a_1)$. There are two cases to consider.
\medskip

\noindent {\em Case 1. k is nonarchimedean.} In the definition of $\psi$, we can choose $\kappa \in k$ with $|\kappa|>1$, $$\psi(a_1)=\left\{\begin{matrix}[1.5] w\textup{diag}\big(\kappa^{d}, \kappa^{d-1},1+\kappa^{d-1}, \ldots, \kappa,1+ \kappa, \frac{1}{\kappa^{d}} \prod_{i}\frac{1}{\kappa^{i}+\kappa^{2i}}\big)w^{-1}, \ n=2d \ \ \  \ \ \  \ \ \ \ \\
w\textup{diag}\big(\kappa^{d}, \kappa^{d-1},1+ \kappa^{d-1}, \ldots, \kappa,1+ \kappa, 1, \frac{1}{\kappa^{d}} \prod_{i}\frac{1}{\kappa^{i}+\kappa^{2i}}\big)w^{-1}, \  n=2d+1 \ \
\end{matrix}\right.$$ By Lemma \ref{perturb1}, choose $|\kappa|>1$ sufficiently large, depending only on the choice of $(x_1,V_1),\ldots,\\ (x_{2m},V_{2m})$, such that (\ref{estimate0}) and (\ref{estimate0'}) holds for any matrix in an open neighborhood $\Omega_1$ of $\psi(a_1)$ in $\mathsf{GL}_n(k)$. Clearly, $\psi(a_1)$ is $1$-biproximal and all of its diagonal entries are distinct. By Lemma \ref{perturb}, we may shrink, if necessary, the open neighbourhood $\Omega_{1}$ such that every matrix $g\in \Omega_{1}$ is conjugate to a diagonal matrix of the form \begin{align*} \textup{diag}&\big( \kappa_1, \ldots,  \kappa_{2d-1}, \ldots, \kappa_n \big)\\  |\kappa_{1}-\kappa^{d}|<\varepsilon, \ldots, &|\kappa-\kappa_{2d-2}|<\varepsilon,  |\kappa+1-\kappa_{2d-1}|<\varepsilon \end{align*} where $|\kappa_{i}|\leq 2$ when $i>2d-1$. Since $|\kappa|>1$, by (\ref{ultrametric}),  $|\kappa_2|=|\kappa_3|=|\kappa|^{d-1}, \ldots, |\kappa_{2d-2}|=|\kappa_{2d-1}|=|\kappa|$. Thus any matrix in $ \Omega_{1}$ cannot be $2q$-proximal for any $1\leq q\leq \frac{n-1}{2}$, hence any representation $\psi'\in \Omega$ is not $2q$-Anosov.

\noindent  {\em Case 2. $k=\mathbb{R}$.} Fix $\xi>0$ transcedental over $\mathbb{Q}(\pi)$  and let $\mathsf{R}_{\xi^m}:=\begin{pmatrix} \cos\xi^m & -\sin \xi^m \\ \sin\xi^m & \cos\xi^m \end{pmatrix},$ $m\in \mathbb{N}$. Let also $$\psi(a_1):=\left\{\begin{matrix}[1.3]
w \textup{diag}\left(\tau, \mathsf{R}_{\xi},\ldots, \mathsf{R}_{\xi^d}, \tau^{-1} \right)w^{-1} \ \ \  \ \ \ \ \ \ \ & n=2d+2  \\
w\textup{diag}\left(\tau, 2\mathsf{R}_{\xi},\ldots, 2\mathsf{R}_{\xi^d},2^{-d}, \tau^{-1}\right)w^{-1} & n=2d+3 \\
\end{matrix}\right.$$ and $\tau>4^n$ be large enough such that (\ref{estimate0}) and (\ref{estimate0'}) hold for a sufficiently small neighborhood $\Omega_1$ of $\psi(a_1)$. By the choice of $\xi$, every eigenvalue of maximum modulus of $\wedge^{2q}\psi(a_1)$ is not real and the same holds for sufficiently small perturbations of $\psi(a_1)$. Thus, by shrinking, if necessary, the open set $\Omega_1$, for every $\psi'\in \Omega$ and $1\leq q\leq \frac{n-1}{2}$, $\wedge^{2q}\psi'(a_1)\in \mathsf{SL}(\wedge^{2q}\mathbb{R}^n)$ cannot be 1-proximal and hence $\psi'$ is not $2q$-Anosov.\end{proof}

%We shall also need the following lemma which provides conditions for the group generated by two subgroups of $\mathsf{GL}_n(k)$ in ping-pong position to be quasi-isometrically embedded. We also refer to \cite{.} and \cite{.} for closely related statements. 

Ping-pong constructions of  free groups have been exhibited in various settings; see for example  \cite{Tits, Ben96, Quint, BG-TA, CPL24}. For the construction of the examples in Theorems \ref{nonArchimedean} and \ref{Archimedean} we use the following elementary ping-pong lemma, variations of which have also been used in previous constructions (see \cite[Lem. 4.2]{CPL24} and \cite[Prop. 6]{DKT25}).

\begin{lemma}\label{pingpong-lem} Let $\Gamma_1,\Gamma_2$ be  finitely generated semigroups and fix \hbox{$|\cdot|_{i}:\Gamma_i\rightarrow \mathbb{R}_{+}$} word metrics on $\Gamma_i$, $i=1,2$. Let $\{\rho_{s'}:\Gamma_1\rightarrow \mathsf{GL}_n(k)\}_{s'\in I}$ and $\{\phi_s:\Gamma_2\rightarrow \mathsf{GL}_n(k)\}_{s\in J}$ be non-empty families of representations such that there exist non-empty subsets $\mathcal{C}_1, \mathcal{C}_2\subset \mathbb{P}(k^d)$ and $c,\theta>0$, with the following properties:\begin{enumerate}\item For every $s'\in I$, $\gamma_1 \in \Gamma_1\smallsetminus \{1\}$ and $[v_1]\in \mathcal{C}_1$, $[\rho_t(\gamma)v_1]\in \mathcal{C}_2$ and $$\big|\big|\rho_{s'}(\gamma)v_1\big|\big|\geq e^{\theta |\gamma|_{1}-c} | | v_1| |.$$
\item For every $s\in J$, $\gamma_2 \in \Gamma_2\smallsetminus \{1\}$ and $[v_2]\in \mathcal{C}_2$, $[\phi_s(\gamma)v_2]\in \mathcal{C}_1$ and $$\big|\big|\phi_s(\gamma)v_2\big|\big|\geq e^{\theta |\gamma|_{2}+c} | | v_2| |.$$ \end{enumerate}

\noindent Then for every $(s',s)\in I\times J$, the unique representation of $\Gamma_1\ast \Gamma_2$ into $\mathsf{GL}_n(k)$, restricting to $\rho_{s'}$ on $\Gamma_1$ and to $\phi_s$ on $\Gamma_2$, is a quasi-isometric embedding. \qed \end{lemma}

%\begin{proof} The proof follows by the sub-multiplicativity of the operator norm. If $\gamma_1\cdots \gamma_n\in \langle \Gamma_1,\Gamma_2\rangle$ is a reduced word, $\gamma_i\in  \Gamma_1\cup \Gamma_2$, $n\geq 1$, then by conditions (i) and (ii), for every $v\in \mathcal{C}_j$, where $j\in \{1,2\}$ is the index such that $\gamma_n\in \mathcal{C}_{3-j}$, we have $||\gamma_1\cdots \gamma_n v||\geq c^{\pm 1} e^{\alpha \sum_{j}|\gamma_i|_{\Gamma}}$.\end{proof}

\section{Proof of Theorem \ref{nonArchimedean}}\label{nonArch}
In this section, $k$ is a nonarchimedean local field. 

\begin{proof}[Proof of Theorem \ref{nonArchimedean}]  
\par Let $A_n\in \mathsf{SL}_{n-3}(k)$ be the empty block if $n=3$, otherwise define $$A_{n}:=\left\{\begin{matrix}[1.5]
\textup{I}_{n-3} &n=4,5  \\
\textup{diag}\Big(\nu_1, 1+\nu_1,\ldots,\nu_s,1+ \nu_s, \prod_i (\nu_i+\nu_i^2)^{-1}\Big)\ \ \ \  & n=2s+4  \\
\textup{diag}\Big(\nu_1,1+\nu_1,\ldots,\nu_s,1+ \nu_s,  \prod_i (\nu_i+\nu_i^2)^{-1},1\Big) & n=2s+5  \\
\end{matrix}\right.$$ where $\nu_1,\ldots,\nu_s\in k$, $|\nu_1|>\cdots>|\nu_s|>2$ are fixed and distinct. Fix also $\nu\in k$ with \hbox{$|\nu|\geq 10n\big(1+||A_n^{\pm 1}||\big)$} and consider the matrix $$C:=\begin{pmatrix}
\nu &  &  &  \\
 &   1+ \nu &  \\
 &  &  \frac{1}{\nu+\nu^2}& \\
 &  &  &  A_n\\
\end{pmatrix}.$$ Note that $C$ has all of its eigenvalues distinct. Fix $0<\epsilon\leq \min \big\{\frac{1}{2}\underset{i\neq j}{\min}|C_{ii}-C_{jj}|, |\nu|^{-3}\big\}$. By Lemma \ref{perturb}, choose the open set $$\Omega_{\epsilon}(C):=\big\{C_{\epsilon}\in \mathsf{GL}_n(k):||C-C_{\epsilon}||<\theta(n,\epsilon, C)\big\}$$ where $\theta(n,\epsilon, C)>0$ is defined in (\ref{constant-0}), \hbox{such that any matrix $C_{\epsilon}\in \Omega_{\epsilon}(C)$ is of the form} \begin{align}\label{Bdef}C_{\epsilon}&:=h\begin{pmatrix}\lambda_1 &  &  &  \\
 &    \lambda_2 &  \\
 &  &   \lambda_3& \\
 &  &  &  A_n'
\end{pmatrix}h^{-1}\\ \nonumber ||h^{\pm 1}-\textup{I}_{n}||< \epsilon,\ |\lambda_1-&\nu|<\epsilon,\  |\lambda_2-(1+\nu)|<\epsilon,\ |\lambda_3-(\nu+\nu^2)^{-1}|<\epsilon,\end{align} and $(A_{n}')^{\pm 1}\in \mathsf{GL}_{n-3}(k)$ is a diagonal matrix with \begin{align}\label{diag-ent}\max_{1\leq j\leq n-3} \big|(A_n)_{jj}-(A_n')_{jj}\big|<\epsilon.\end{align}  Since $|\nu|>1$, by (\ref{ultrametric}), $|\lambda_1|=|\lambda_2|=|\nu|$ and $C_{\epsilon}\in \Omega_{\epsilon}(C)$ fails to be a $1$-proximal. If $n\geq 6$, by using the fact that $|x_j|>1$ for every $j$, (\ref{diag-ent}) and (\ref{ultrametric}), the first $2s$ diagonal entries of $A_n'$ have moduli $|\nu_1|,|\nu_1|,\ldots,|\nu_s|,|\nu_s|\geq 2$. Thus, the moduli of the first $2s+2$ eigenvalues of the matrix $C_{\epsilon}\in \Omega_{\epsilon}(C)$ are $|\nu|,|\nu|, |\nu_1|,|\nu_1|, \ldots, |\nu_s|,|\nu_s|.$ In particular, any matrix in $ \Omega_{\epsilon}(C)$ fails to be $(2q+1)$-biproximal for any $0\leq q\leq \frac{n-2}{2}$.

Now consider the flag $(y_n,V_n)\in \mathcal{F}_{1,n-1}(k)$,  \begin{align*}\big(y_n,V_n\big):=\big([ e_1+\nu e_2-(1+\nu)e_3], (e_1+e_2+e_3)^{\perp}\big).\end{align*} \par By Lemma \ref{perturb2}, there is a representation $\psi_{m-1}:\mathsf{F}_{m-1}\rightarrow\mathsf{SL}_{n}(k)$ and an open neighborhood $\Omega_{0}\subset \textup{Hom}(\mathsf{F}_{m-1},\mathsf{GL}_n(k))$ of $\psi_{m-1}$, with the property that for every $\psi_{m-1}'\in \Omega_{0}$, \begin{align}\label{pingpong-1}  \forall  w\in \mathsf{F}_{m-1}\smallsetminus \{1\}, \ & \psi_{m-1}'(w)\big(\mathbb{P}(k^n)\smallsetminus \mathcal{N}_{\epsilon}(\mathbb{P}(V_n))\big)\subset B_{\epsilon}(y_n)\\  
\label{pingpong-2} \forall    [v]\in \mathbb{P}(k^n)\smallsetminus \mathcal{N}_{\epsilon}&(\mathbb{P}(V_n)),  \   || \psi'_{m-1}(w)v||\geq 2^{|w|} ||v||\end{align} In addition, for $n\geq 4$, $\Omega_{0}$ can be chosen such that for any $\psi'_{m-1}\in \Omega_{0}$, $\psi'_{m-1}(\mathsf{F}_{m-1})$ contains a matrix which is not $2d$-proximal for any $1\leq d \leq \frac{n-1}{2}$. Consider the unique representation $\rho: \mathsf{F}_m\rightarrow \mathsf{SL}_n(k)$, $\mathsf{F}_{m}= \langle a_1\rangle \ast \mathsf{F}_{m-1}$, satisfying \begin{align}\label{rhodef-1}\rho(a_1)&=\textup{diag}\big(\nu, 1+\nu,(\nu+\nu^2)^{-1},A_n\big)\\ \label{rhodef-2}\rho(a_i)&=\psi_{m-1}(a_i), \ i=2,\ldots,m.\end{align} 

\begin{claim}\label{claim1} For  every $p\neq 0$ and $C_{\epsilon}\in \Omega_{\epsilon}(C)$ as in (\ref{Bdef}), $$C_{\epsilon}^pB_{\epsilon}(y_n)\subset \mathbb{P}(k^n)\smallsetminus \mathcal{N}_{\frac{1}{6}}(\mathbb{P}(V_n)).$$ In addition, for every $[u]\in B_{\epsilon}(x)$ we have $||C_{\epsilon}^pu||\geq 2^{|p|} ||u ||.$\end{claim}

\begin{proof} Fix $[u]\in B_{\varepsilon}(y_n)$. Up to multiplying $u\in k^n$ with a scalar, we may assume that $$\big|\big|u-(e_1+\nu e_2-(1+\nu)e_3)\big|\big|\leq \epsilon |\nu|<|\nu|^{-2}.$$ By (\ref{ultrametric}) we have $||u||=|\nu|$. Given $C_{\epsilon}\in \Omega_{\epsilon}(C)$, write $C_{\epsilon}=h\Delta h^{-1}$ as in (\ref{Bdef}), where $\Delta:=\textup{diag}(\lambda_1,\lambda_2,\lambda_3,A_n')$ and \hbox{$||h^{\pm 1}-\textup{I}_n||< \epsilon$, $|\lambda_1|=|\lambda_2|=|\nu|$, $|\lambda_3-\nu^{-2}|<\epsilon$,\hbox{$\big|\frac{\lambda_3}{\lambda_1}\big|\leq 2|\nu|^{-3}$.} Thus,} \begin{align} \label{diagbound} \big|\big|C_{\epsilon}^p u-\Delta^p( e_1+\nu e_2-(1+\nu)e_3)\big|\big|&\leq 6\epsilon ||\Delta ^p|| \cdot | \nu |. \end{align}

%and write $u=(c_1,c_2,c_3,\omega)$ \hbox{for some $c_1,c_2,c_3\in k, \omega \in k^{n-3}$ with} $$\max\Big\{ |c_1-1|, |c_2-\nu|, |c_3+(1+\nu)|, | |\omega| |\Big\}<\epsilon|\nu|<|\nu|^{-2}.$$ By (\ref{ultrametric}) we have $|c_1|=1$, $|c_2|=|\nu|$, $|c_3|=|1+\nu|=|\nu|$ and $||u||=|\nu|$.

% \nonumber \big|\big|h\Delta^ph^{-1}u-\Delta^ph^{-1}u\big|\big|&\leq \big|\big|h-\textup{I}_d\big|\big|\cdot \big|\big|\Delta^p\big|\big|\cdot \big|\big|u\big|\big|\cdot \big|\big| h^{-1}\big|\big|\\ \nonumber &\leq 2\epsilon \big|\big|\Delta ^p\big|\big|\cdot  \big|\big|u\big|\big|,\\ \nonumber \big|\big|\Delta^ph^{-1}u-\Delta^pu\big|\big|&\leq \big|\big|h^{-1}-\textup{I}_d\big|\big|\cdot \big|\big|\Delta^p\big|\big|\cdot \big|\big|u\big|\big|\\ \nonumber &\leq \epsilon  \big|\big|\Delta ^p\big|\big|  \cdot \big|\big|u\big|\big|,\\ 

\noindent {\em Case 1: $p>0$.} Then $||\Delta^p||=\max\{|\lambda_1|^p, |\lambda_3|^p, ||(A_n')^p||\}=|\nu|^p$, since $||(A_n')^p||\leq |\nu|^p$.  In particular, $||C_{\epsilon}^p||\leq ||h||\cdot ||h^{-1}||\cdot |\lambda_1|^p\leq 4|\lambda_1|^{p}$. By using (\ref{diagbound}), we obtain the bounds \begin{align*} \frac{| C_{\epsilon}^pu \cdot (e_1+e_2+e_3) |}{||u||\cdot ||C_{\epsilon}||^p}&\geq \frac{\big|\lambda_1^p+ \lambda_2^p \nu-\lambda_3^p (1+\nu)\big|}{4|\nu| \cdot|\lambda_1|^p} -6\epsilon\\ &\geq \frac{1}{4|\nu|}\big( |\nu|-1-2|\nu|^{-2}\big)-6\epsilon> \frac{1}{6},\end{align*} hence  by (\ref{distance}) we have $[C_{\epsilon}^pu]\in \mathbb{P}(k^n)\smallsetminus \mathcal{N}_{\frac{1}{6}}(\mathbb{P}(V_n))$. 

%&=\frac{1}{|\nu|}\Bigg|c_1+ \frac{\lambda_2^p}{\lambda_1^p}c_2+\frac{\lambda_3^p}{\lambda_1^p}c_3+\frac{1}{\lambda_1^{p}}\sum_{j}(A_{n}')_{jj}^p\omega_j  \Bigg|-3\epsilon\\
\medskip

\noindent {\em Case 2: $p<0$.}  Then $||\Delta^{p}||=\max\big \{|\lambda_1|^{p}, |\lambda_3|^{p}, ||(A_n')^{p}||\big \}=|\lambda_3|^{p}$, since $|\lambda_3|\leq 2|\nu|^{-2}$. In particular, $||C_{\epsilon}^p||\leq 4|\lambda_3|^{p}$ and $\big|\frac{\lambda_3}{\lambda_i}\big|\leq 2|\nu|^{-3}$ for $i=1,2$. By aguing as in the first case and using (\ref{diagbound}), we obtain the bounds \begin{align*} \frac{\big| C_{\epsilon}^p u \cdot (e_1+e_2+e_3)\big|}{||u||\cdot ||C_{\epsilon}^p||}&\geq \frac{ \big|\lambda_1^p+ \lambda_2^p \nu-(1+\nu)\lambda_3^p \big|}{4|\nu|\cdot |\lambda_3^p|}-6\epsilon\geq \frac{1}{4}-\frac{2^{|p|-1}}{|\nu|^{3|p|}}-6\epsilon> \frac{1}{6}.
\end{align*} and hence, by (\ref{distance}), we have that $[C_{\epsilon}^pu]\in \mathbb{P}(k^n)\smallsetminus \mathcal{N}_{\frac{1}{6}}(\mathbb{P}(V_n))$. 

In addition, by (\ref{diagbound}), for every $p\in \mathbb{Z}^{\ast}$, $$\frac{||C_{\epsilon}^pu||}{||u||}\geq (1-6\epsilon)||\Delta^p||\geq (1-6\epsilon)|\nu|^{|p|}\geq \sqrt[]{|\nu|}^{|p|}.\qedhere$$ Thus the claim follows.\end{proof} Now we complete the proof of the theorem. Clearly, $$\Omega_{\epsilon}(C)\times \Omega_{0}=\Big\{\rho'\in \mathsf{GL}_n(k)\times \textup{Hom}(\mathsf{F}_{m-1},\mathsf{GL}_n(k)):\rho'(a_0)\in \Omega_{\epsilon}(C),\rho'|_{\mathsf{F}_{m-1}}\in \Omega_{0}\Big\}$$ is an open neighbourhood of the representation $\rho:\mathsf{F}_{m}\rightarrow \mathsf{GL}_n(k)$ defined by (\ref{rhodef-1}) and (\ref{rhodef-2}). By Claim \ref{claim1}, (\ref{pingpong-1}) and (\ref{pingpong-2}), the sets $\mathcal{C}_1:=B_{\epsilon}(y_n)$, $\mathcal{C}_2:=\mathbb{P}(k^n)\smallsetminus \mathcal{N}_{\epsilon}(\mathbb{P}(V_n))$ and the representations $\Omega_{\epsilon}(C)\subset \textup{Hom}(\mathbb{Z},\mathsf{GL}_n(k))$, $\Omega_{0}\subset \textup{Hom}(\mathsf{F}_{m-1},\mathsf{GL}_n(k))$ satisfy the assumptions of Lemma \ref{pingpong-lem}. Thus, any representation $\rho'\in \Omega_{\epsilon}(C)\times \Omega_{0}$ is a quasi-isometric embedding of $\mathsf{F}_m$. Moreover, $\rho'$ is not Anosov since $\rho'(a_1)\in \Omega_{\epsilon}(C)$ is not $(2q+1)$-proximal for any $q$, while $\rho'(\mathsf{F}_{m-1})$ contains a non-trivial element which is not $2d$-proximal for any $d$. This completes the proof of the theorem. \end{proof} 

\section{Proof of Theorem \ref{Archimedean}}\label{Arch}

%\begin{theorem}\label{Archimedean-det} Let $F$ be a finite group and $\psi_0:F\rightarrow \mathsf{GL}_d(\mathbb{R})$ be a faithful representation, of even dimension \hbox{$d$}, acting irreducibly on $\mathbb{C}^d$. For $m\geq 1$, consider the virtually free group $$\Gamma_{F}^m:= \big(\mathbb{Z}\times F\big)\ast \mathsf{F}_{m}.$$ For any $n\geq d+1$, there exists a representation $\psi_n:\Gamma\rightarrow \mathsf{GL}_d(\mathbb{R})$ and an open neighborhood $\mathcal{O}_n$ of $\psi_n$ in $\textup{Hom}(\Gamma_{F}^m ,\mathsf{GL}_n(\mathbb{C}))$ with the following properties:\begin{enumerate}
% \item \label{Arch-det-1} Every representation in $\mathcal{O}_{n}$ is a quasi-isometric embedding and non-locally rigid;
%\item \label{Arch-det-2} $\mathcal{O}_{n} \cap \textup{Anosov}_j\left(\Gamma,\mathsf{GL}_n(\mathbb{C})\right)$ is empty for any $j=1,\ldots,d-1$; and 
%\item \label{Arch-det-3} $\mathcal{O}_{n} \cap \textup{Anosov}_r\left(\Gamma,\mathsf{GL}_n(\mathbb{R})\right)$ is empty for any $r=1,\ldots,n-1$.\end{enumerate}\end{theorem}

\begin{proof}[Proof of Theorem \ref{Archimedean}] Given $d\geq 2$, fix $F$ a finite group and $\psi_0:F\rightarrow \mathsf{O}(d)$ a faithful representation such that $\psi_0(F)$ acts irreducibly on $\mathbb{C}^d$. Choose unit vectors $v_1,v_2\in \mathbb{R}^d$ and $\eta>0$ such that \begin{align}\label{v1-ineq}\min_{f\in F}\big|\psi_0(f)v_1\cdot v_2\big|\geq \eta, \ \min_{f\in F\smallsetminus \{1\}}\big|\big(\psi_0(f)v_1-v_1\big)\cdot v_2\big|\geq \eta.\end{align} Let $t\in \mathbb{Z}\times F$ be the generator of $\mathbb{Z}$. For $n\geq d+1$, define the representation \hbox{$\psi_{n}:\mathbb{Z}\times F \rightarrow \mathsf{GL}_n(\mathbb{R})$} for $f\in F$ and $t\in \mathbb{Z}$ as follows: \begin{align}\label{rho1-def}\psi_{n}(f)=\begin{pmatrix}
\psi_0(f) &  &   \\
&   1  &  \\
 &  &   \textup{I}_{n-d-1} \\
\end{pmatrix}  \ \ \psi_{n}(t)=\begin{pmatrix}
\lambda \textup{I}_d & &   \\
 & \lambda^{-d} &    \\
 &  &   \Theta_n \\
\end{pmatrix} \end{align} where $\lambda\geq \frac{8^n}{\eta^2}$ is fixed. Here $\Theta_n$ is the empty block if $n=d+1$, otherwise \begin{align}\label{Theta-n}\Theta_{n}:=\left\{\begin{matrix}[1.2]
1   &n=d+2 \\
\textup{diag}\big(\mathsf{R}_{\xi},\ldots, \mathsf{R}_{\xi^s}   \big)  &n=2s+d+1  \\
\textup{diag}\big(2 \mathsf{R}_{\xi},\ldots, 2 \mathsf{R}_{\xi^s},4^{-s}  \big)   &n=2s+d+2 \\
\end{matrix}\right.\end{align} where $\xi>0$ is  transcedental over $\mathbb{Q}(\pi)$ and $\mathsf{R}_{\xi^m}\in \mathsf{SO}(2)$ is the rotation with respect to $\xi^m$.

\begin{lemma}\label{claim3} For every $0<\theta<1$, there exists $\theta'>0$, depending on $\theta$, such that $$\Psi_{\theta'}:=\left\{\phi\in \textup{Hom}\big(\mathbb{Z}\times F,\mathsf{GL}_n(\mathbb{C})\big):  \begin{matrix}[1.2] \forall  f\in F, \  \phi(f)=w\psi_{n}(f)w^{-1}\\
 ||w-\textup{I}_n||<\theta',  ||\phi(t)-\psi_{n}(t)||<\theta' \end{matrix}  \right\}$$ is an open neighborhood of $\psi_{n}$ and for any $\phi \in \Psi_{\theta'}$, there is $w_{\phi}\in \mathsf{GL}_n(\mathbb{C})$ with $| |w_{\phi}^{\pm 1}-\textup{I}_n| |<\theta$, \begin{align*}\phi(t)=w_\phi\begin{pmatrix}
\lambda_1 \textup{I}_d & &   \\
&   \lambda_{2} &  \\
 &  &   \Theta_n' \\
\end{pmatrix}w_{\phi}^{-1}, \ \phi(f)=w_{\phi}\psi_n(f)w_{\phi}^{-1} \ \forall f \in F,\end{align*} where $|\lambda_1-\lambda|<\theta$, $|\lambda_2-\lambda^{-d}|<\theta$ and $| |\Theta_n'-\Theta_n | |<\theta$. \hbox{If $n=d+1$ we can take $\theta'=\frac{\theta}{10^4 ||\psi_n(t)||}$.} \end{lemma}

\begin{proof} Since $\textup{diag}(\lambda^{-d}, \Theta_n)$ has all of its eigenvalues distinct, there is $\theta_0>0$ such that any matrix in $\mathsf{GL}_{n-d}(\mathbb{C})$, within the $\theta_0$-ball of $\textup{diag}(\lambda^{-d}, \Theta_n)$, is of the form $h_0 \textup{diag}(\lambda_0, \Theta_n')h_0^{-1}$ where $||h_0-\textup{I}_{n-d}||<10^{-2}\theta$, $|\lambda_0-\lambda^{-d}|<\theta$ and $||\Theta_n- \Theta_n'||<\theta$. Now take $\theta':=\frac{\theta_0}{10^4 ||\psi_n(t)||}$; if $n=d+1$ (i.e. $\Theta_n$ is the empty block) take $\theta':=\frac{\theta}{10^{4}||\psi_n(t)||}$.

Since $F$ is finite, the restriction $\psi_{n}|_{F}:F\rightarrow \mathsf{GL}_n(\mathbb{C})$ is locally rigid\footnote{Given $\zeta>0$, there is $\zeta'>0$ such that any representation $\pi:F\rightarrow \mathsf{GL}_n(\mathbb{C})$ with $||\pi(f)-\psi_n(f)||<\zeta'$ for every $f\in F$, there exists $b\in \mathsf{GL}_n(\mathbb{C})$ with $||b-\textup{I}_n||<\zeta$ and $\pi(f)=b\psi_n(f)b^{-1}$ for every $f\in F$.} \cite{Weil} and $\Psi_{\theta'}$ is an open neighborhood of $\psi_n$. Fix $\phi\in \Psi_{\theta'}$ with $\phi|_F=w(\psi_n|_F)w^{-1}$ for some $||w-\textup{I}_n||<\theta'$. Since $w^{-1}\phi(t) w$ centralizes $\psi_n(F)$ and $\psi_0(F)\subset \mathsf{GL}_d(\mathbb{C})$ is irreducible, we have $$\phi(t)=w\psi_n'(t)w^{-1}, \ \psi_n'(t):=\textup{diag}\big(\lambda_1 \textup{I}_d,  \Theta_{n1}\big)$$ for some $\lambda_1\in  \mathbb{C}$ and $\Theta_{n1}\in \mathsf{GL}_{n-d}(\mathbb{C})$. Then, note that \begin{align*} | | \psi_n'(t)-\psi_n(t) | | & \leq 4\theta'+4| | w-\textup{I}_n| |\cdot | | \psi_n(t)| |+4 | |w | |\cdot | |\psi_n(t)| |\cdot | |w^{-1}-\textup{I}_n||<\theta_1.  \end{align*}

%\leq |  w | | \cdot | | w^{-1} | |\cdot  | | \phi(t)-w \psi_n(t)w^{-1}| |\\ &

 This implies $|\lambda_1-\lambda|<\theta$ and $||\Theta_{n1}-\textup{diag}(\lambda^{-d},\Theta_n)||<\theta_1$. By the choice of $\theta_1>0$, there are $h_0\in \mathsf{GL}_{n-d}(\mathbb{C})$, $||h-\textup{I}_n||<10^{-2}\theta$, and $\Theta_n'\in \mathsf{GL}_{n-d-1}(\mathbb{C})$ with $\Theta_{n,1}=h_0\textup{diag}(\lambda_2, \Theta_n')h_0^{-1}$, $||\Theta_n'-\Theta_n||<\theta$ and $|\lambda_2-\lambda^{-d}|<\theta$. Finally, we \hbox{take $w_\phi:=w\textup{diag}(1,h_0)$.}\end{proof} 

Consider  the flag $(y,W)\in \mathcal{F}_{1,n-1}(\mathbb{C})$, \begin{align*} (y,V)=\big([v_1+e_{d+1}],\big(v_2 -(v_1\cdot v_2)e_{d+1} \big)^{\perp}\big).\end{align*} Given $0<\eta\leq1$ in (\ref{v1-ineq}), let $\theta:=10^{-2}\eta$ and fix  $\Psi_{\theta'}\subset \textup{Hom}(\mathbb{Z}\times F,\mathsf{GL}_n(\mathbb{C}))$ the neighborhood of $\psi_{n}:\mathbb{Z}\times F\rightarrow \mathsf{SL}_n^{\pm 1}(\mathbb{R})$ defined in Claim \ref{claim3}. 

\begin{claim} \label{claim4}For any $\psi_{n}'\in \Psi_{\theta'}$, $f\in F$ and $p\in \mathbb{Z}$ with $ft^p\neq 1$, $$\psi_{n}'(ft^p)B_{\theta}(y)\subset \mathbb{P}(\mathbb{C}^n)\smallsetminus \mathcal{N}_{\theta}(\mathbb{P}(W))$$ and for every $[u]\in B_{\epsilon}(y)$, $\big|\big|\psi_{n}'(ft^p)u\big| \big|\geq \frac{\theta}{10}\lambda^{|p|} ||u||$.\end{claim}

\begin{proof} By Claim \ref{claim3}, write $\psi_{n}'(t)=w_0\textup{diag}\big(\lambda_1\textup{I}_d,\lambda_2, \Theta_n' \big)w_0^{-1}$, \hbox{$\psi_{n}'(f)=w_0\psi_{n}(f)w_0^{-1}$ for $f\in F$,} $$|\lambda_1-\lambda|<\theta, \ |\lambda_2|\leq \lambda^{-d}+\theta\leq 2\theta, \ | |w_0^{\pm 1}-\textup{I}_n| |<\theta$$ and $||\Theta_n-\Theta_n'||<\theta,  ||(\Theta_n')^{\pm 1}||\leq 5^n.$

Fix $[u]\in B_{\theta}(y)$ such that $||u-(v_1+e_{d+1})||\leq 2\theta$ and $||w_0^{-1}u-(v_1+e_{d+1})||\leq 6\theta$. Note also that $||w_0^tv_2-(v_1\cdot v_2)w_0^t e_{d+1}-(v_2-(v_1\cdot v_2)e_{d+1}||\leq 6 \theta$. There are two cases to consider.
\medskip

%By Lemma \ref{dist}, as $||w_0^{\pm 1}||\leq 2$ and $C(w_0)\leq 8$, it follows that $w_0^{-1}B_{\theta}(y)\subset B_{20\theta}(y)$. Fix $[u]\in B_{\theta}(y)$ and write $w_0^{-1}u=[u_1+\mu e_{d+1}+\omega]$, where $u_1\in \mathbb{R}^d$, $|u_1-v_1|<20\theta$, $|\mu-1|<20\theta$ and if $n\geq d+2$, $\omega \in \mathbb{R}^{n-d-1}$ and $||\omega||\leq 20\theta$.

\noindent {\em Case 1: $p= 0$ and $f\in F\smallsetminus\{1\}$ or $p>0$ and $f\in F$.} Since $\lambda_1>\frac{\lambda}{2}$, $\frac{\lambda_2}{\lambda_1}<\frac{1}{\lambda}$ and $\lambda>\frac{8^n}{\eta^{2}}$, by using (\ref{v1-ineq}), since $||\psi_n'(ft^p)||\leq ||w_0||\cdot||w_0^{-1}||\cdot |\lambda_1|^p\leq 4|\lambda_1|^p$, we have the bounds,
\begin{align*} \frac{\big|\psi_{n}'(ft^p)u\cdot \big(v_2-(v_1\cdot v_2)e_{d+1}\big) \big|}{||\psi_n'(ft^p)||} &\geq \frac{1}{4|\lambda_1|^p}\Big| \textup{diag}\big(\lambda_1^p \psi_0(f),\lambda_2^p, (\Theta_n')^p \big)(w_0^{-1}u)\cdot \big(w_0^{t}v_2-(v_1\cdot v_2) w_0^t e_{d+1}\big)\Big|\\
& \geq \frac{1}{4|\lambda_1|^p}\Big| \textup{diag}\big(\lambda_1^p\psi_0(f),\lambda_2^p, (\Theta_n')^p \big)(w_0^{-1}u)\cdot \big(v_2-(v_1\cdot v_2) e_{d+1}\big)\Big|-2\theta \\  & \geq \frac{1}{4}\Big| \big(\psi_0(f)v_1\cdot v_2\big)-\frac{\lambda_2^p}{\lambda_1^p}(v_1\cdot v_2) \Big| -8\theta \geq \frac{\eta}{4}-10\theta. \end{align*} 

%&= \frac{1}{4}\Big| \big(\psi_0(f)u_1\cdot v_2\big)-\mu \frac{\lambda_2^p}{\lambda_1^p} (v_1\cdot v_2) \Big| -5\theta \\ 

\par Hence $\big|\psi_{n}'(ft^p)u\cdot (v_2-(v_1\cdot v_2) w_0^t e_{d+1}) \big| \geq 10\theta | |\psi_{n}'(ft^p)|| $,  $||\psi_{n}'(ft^p)u|| \geq \frac{\theta}{4}\lambda^p ||u||$ and $[\psi_{n}'(ft^p)u]\in \mathbb{P}(\mathbb{C}^n)\smallsetminus \mathcal{N}_{\theta}(\mathbb{P}(W))$.
\medskip

\noindent {\em Case 2: $p<0$}. We have $|\lambda_2|^{-|p|}\geq (2\theta)^{-|p|}\geq \max\{|\lambda_1|^{-|p|}, ||(\Theta_n')^p||\}$,  $||\psi_n'(ft^p)||\leq 4|\lambda_2|^{p}$. By working similarly as in Case 1 we have the bounds, \begin{align*} \frac{\big|\psi_{n}'(ft^p)u\cdot \big(v_2-(v_1\cdot v_2)e_{d+1}\big) \big|}{||\psi_n'(ft^p)||} &\geq \frac{1}{4|\lambda_2|^p}\Big| \textup{diag}\big(\lambda_1^p\psi_0(f),\lambda_2^p, (\Theta_n')^p \big)(w_0^{-1}u)\cdot \big(w_0^{t}v_2-(v_1\cdot v_2) w_0^t e_{d+1}\big)\Big|\\
& \geq  \frac{1}{4}\Big| \frac{\lambda_1^{p}}{\lambda_2^{p}}\big(\psi_0(f)u_1\cdot v_2\big)- (v_1\cdot v_2) \Big| -5\theta\\ & \geq \frac{1}{4}\big|(v_1\cdot v_2) \big| -5\theta -\frac{| |\psi_0(f)u_1||}{4\lambda}\geq    10\theta. \end{align*} Therefore, $[\psi_{n}'(g)u]\in \mathbb{P}(\mathbb{C}^n)\smallsetminus \mathcal{N}_{\theta}(\mathbb{P}(W))$ and $| |\psi_{n}'(ft^p)u| | \geq 5\theta | |\psi_{n}'(ft^p)| |\geq \frac{\theta}{2}\lambda^{|p|}||u||$.\end{proof}

 By Lemma \ref{perturb2} there is a representation $\varphi:\mathsf{F}_m\rightarrow \mathsf{SL}_n(\mathbb{R})$ and an open neighborhood $\Omega_n$ of $\varphi$, such that for any $\varphi'\in \Omega_n$ and $g\in \mathsf{F}_m\smallsetminus \{1\}$ the following hold:
\begin{itemize}
\item $\varphi'(g)\big(\mathbb{P}(\mathbb{C}^n)\smallsetminus \mathcal{N}_{\theta}(\mathbb{P}(W))\big)\subset B_{\theta}(y)$;
\item $| | \varphi'(g) u| |\geq 2^{|g|}\theta^{-6}  | |u| |$ for any $[u]\in B_{\epsilon}(y)$; and 
\item if $r\in \{1,\ldots,n-1\}$ is even, \hbox{$\Omega \cap \textup{Anosov}_r\left(\mathsf{F}_m,\mathsf{GL}_n(\mathbb{R})\right)$ is empty.}\end{itemize}

Let $\Gamma:= \big(\mathbb{Z}\times F\big)\ast \mathsf{F}_{m}$ and $\psi_n: \Gamma \rightarrow \mathsf{SL}_n^{\pm}(\mathbb{R})$ be the unique representation with $\psi_n|_{\mathbb{Z}\times F}$ defined in (\ref{rho1-def}) and $\psi_n|_{\mathsf{F}_m}=\varphi$. Claim \ref{claim4} and Lemma \ref{pingpong-lem}, applied for the ping-pong sets $\mathcal{C}_1:=B_{\theta}(y)$, $\mathcal{C}_2:=\mathbb{P}(\mathbb{C}^n)\smallsetminus \mathcal{N}_{\theta}(\mathbb{P}(W))$ and the representations $\Psi_{\theta'}$ and $\Omega_n$ respectively, shows that $\Psi_{\theta'} \times \Omega_n$  contains entirely of quasi-isometric embeddings of $\Gamma$ into $\mathsf{GL}_n(\mathbb{C})$.

By the choice of $\varphi$ and $\psi_n|_{\mathbb{Z}\times \mathsf{D}_8}$ in (\ref{rho1-def}), for every $\psi_n'\in \Psi_{\theta'}$, $\psi_n'(t)$ is not $\{1,\ldots,d-1\}$-proximal, thus $(\Psi_{\theta'}\times \Omega_n)\cap \textup{Anosov}_j(\Gamma,\mathsf{GL}_n(\mathbb{C}))$ is empty when $j=1,\ldots,d-1$. Since $d$ is even, $\psi_n(t)$ is not $r$-proximal when $r$ is odd. Moreover, by the choice of $\varphi$, since a generator $\varphi(\mathsf{F}_m)$ is not $r$-biproximal when $r$ is even, we deduce that $\psi_n:\Gamma \rightarrow \mathsf{GL}_n(\mathbb{R})$ is not Anosov.

Now we check that (\ref{Arch-det-3}) holds. If $r$ is even this is clear by the choice of $\varphi$. If $n\geq d+1$, the eigenvalues of $\psi_{n}(t)$ are $$ \begin{matrix}[1] \lambda,\ldots, \lambda, e^{i\xi},e^{-i\xi},\ldots, e^{i\xi^s},e^{-i\xi^s}, \lambda^{-d} \ \ \ \ \ \ \ \ \ \ \ \ \  & n=2s+d+1  \\
\lambda, \ldots,\lambda, 2e^{i\xi},2e^{-i\xi},\ldots, 2e^{i\xi^s},2e^{-i\xi^s},4^{-s}, \lambda^{-d} & n=2s+d+2. \end{matrix}$$ As $\xi$ is transcedental over $\mathbb{Q}(\pi)$ and $d$ is odd, for any $d+1\leq d+q\leq n-2$ odd integer, the matrix  $\wedge^{d+q} \psi_{n}(t)$ has all of its eigenvalues of maximum modulus in $\mathbb{C}\smallsetminus \mathbb{R}$. If $n>d+1$, by shrinking if necessary the open set $\Psi_{\theta'}$, we ensure that $\wedge^{q+d}\psi_{n}'(t)$ is not $1$-proximal for any $\psi_n'\in \Psi_{\theta'}$. This shows that $\psi_n'$ is not $(d+q)$-Anosov when $d+q$ is odd. In conclusion, the open neighborhood $\mathcal{O}_n:=\Psi_{\theta'}\times \Omega_n$ of $\psi_n$ \hbox{satisfies (\ref{Arch-det-1}), (\ref{Arch-det-2}) and (\ref{Arch-det-3}).} \end{proof}

\begin{rmk}\label{rmk-378}\normalfont{Examples of Zariski dense, non-Anosov, robust quasi-isometric embeddings of (non-virtually free) hyperbolic groups into $\mathsf{SL}_n(\mathbb{R})$ were known at least when $n\geq 756$. For this, fix $\Delta$ a cocompact lattice in $\mathsf{Sp}(2,1)$ and the irreducible proximal representation $\mathsf{Sp}(2,1)\rightarrow \mathsf{SL}_{14}(\mathbb{C})$, as in the proof of \cite[Cor. 1.2]{Tso-robust}. By extending scalars $\mathsf{SL}_{14}(\mathbb{C})\rightarrow \mathsf{SL}_{28}(\mathbb{R})$ and composing with the exterior power $\wedge^2:\mathsf{SL}_{28}(\mathbb{R})\rightarrow \mathsf{SL}(\wedge^2 \mathbb{R}^{28})$, one obtains a representation $\chi:\mathsf{Sp}(2,1)\rightarrow \mathsf{GL}_{378}(\mathbb{R})$ which is $m$-proximal, $1\leq m \leq 189$, if and only if $m\in \{1,17,61,141\}$. As the restriction $\chi|_{\Delta}$ is $m$-Anosov for the aforementioned odd values of $m\leq 189$ (see \cite[Prop. 4.4]{GW}), the proof of \cite[Thm. 1.1 \& 3.1]{Tso-robust}, applied for $\chi|_{\Delta}$ and $p=2$, provides for any $n\geq 756$ a robust quasi-isometric embedding of a free product $\Delta_{1n}\ast \Delta_{2n}$ (where $\Delta_{in}$ are finite-index subgroups of $\Delta$) into $\mathsf{SL}_{n}(\mathbb{R})$ which is not in the closure of Anosov representations.}\end{rmk}

\section{Robust quasi-isometrically embedded free semigroups of $\mathsf{GL}_3(\mathbb{C})$} \label{Sem-C}  In this section we prove Theorem \ref{Archimedean-sem}. Throughout this section $k=\mathbb{R}$ or $\mathbb{C}$. 

\begin{proposition}\label{svg-est} Let $x:=[x_1e_1+\cdots+x_ne_n]\in \mathbb{P}(k^n)$ be a unit vector and $\epsilon>0$ with $\min_i |x_i|\geq 10^2\epsilon$. Fix $V\in \mathsf{Gr}_{n-1}(k^n)$ and suppose that $\mathcal{A}\subset \mathsf{GL}_n(k)$ is a non-empty set of diagonal matrices with the property that for any $g\in \mathcal{A}$ and $p\in \mathbb{Z}^{\ast}$, $$g^pB_{\epsilon}(x)\subset \mathbb{P}(k^n)\smallsetminus \mathcal{N}_{\epsilon}(\mathbb{P}(V)).$$ Given $M>1$, define the subset of $\mathsf{GL}_n(k)$, \begin{align}\label{B(x,V)}\mathcal{B}_{M,\epsilon}(x,V):=\left\{ h\begin{pmatrix}\kappa_1 & & \\ & \ddots & \\ & & \kappa_n \end{pmatrix}h^{-1}:  \begin{matrix}[1.3] ||h^{\pm 1}||\leq M,  \ [he_1]\in B_{\frac{\epsilon}{10}}(x)\\ \mathbb{P}(he_1^{\perp})\subset \mathcal{N}_{\frac{\epsilon}{10}}(\mathbb{P}(V)) \\  |\kappa_1|\geq \left(\frac{10M}{\epsilon}\right)^{12} |\kappa_i|, \ i=2,\ldots,n \end{matrix}  \right\}.\end{align}  Any reduced product of the form $g_{r}^{p_r}w_{r}^{s_{r}}\cdots g_1^{p_1}w_1^{s_1}$, where $r\geq 2$, $p_1,\ldots,p_r\in \mathbb{Z}$, $s_1,\ldots,s_r\in \mathbb{N}\cup \{0\}$, $g_1,\ldots,g_r\in \mathcal{A}$, $w_1,\ldots,w_{r}\in \mathcal{B}_{M,\epsilon}(x,V)$ satisfies the estimate \begin{align}\label{inequality}\frac{\sigma_1( w_{r}^{s_r}g_{r}^{p_{r}}\cdots w_1^{s_1}g_1^{p_1})}{\sigma_2( w_{r}^{s_r}g_{r}^{p_{r}}\cdots w_1^{s_1}g_1^{p_1})}\geq \epsilon^{5}\prod_{i=1}^{r} \left(\sqrt[3]{\frac{\ell_1(w_i^{s_i})}{\ell_2(w_i^{s_i})}} \frac{\ell_1(g_i^{p_i})}{\ell_2(g_i^{p_i})}\right).\end{align}\end{proposition}

\begin{proof} We shall follow the proof of \cite[Thm. 1.2 (ii) \& (iii)]{Weisman-T}. However, note that the estimates of \cite[Thm. 1.2]{Weisman-T} do not directly imply (\ref{inequality}) since we do not assume that $\mathcal{A}$ contains of $1$-biproximal matrices. We will need the following claims.

\begin{claim}\label{claim-ball1} Fix $y\in B_{\frac{\epsilon}{2}}(x)$ and $0<\zeta<\frac{\epsilon^4}{10^2}$. For any $g\in \mathcal{A}$ and $p\in \mathbb{Z}^{\ast}$ we have that $$g^pB_{\zeta}(y)\subset B_{\frac{\zeta \ell_2(g^p)}{\epsilon^2\ell_1(g^p)}}(g^py).$$\end{claim}

\begin{proof}[Proof of Claim \ref{claim-ball1}] Fix $y=[y_1e_1+\cdots+y_ne_n]$, where $\sum y_i^2=1$ and $y\in B_{\frac{\epsilon}{2}}(x)$, and note that $\min_i |y_i|\geq \epsilon$. Fix also $0<\zeta<\frac{\epsilon^4}{10^2}$ and $y'\in B_{\zeta}(y)$. For any $g\in \mathcal{A}$ and $p\in \mathbb{Z}^{\ast}$, by \cite[Lem. 2.1]{Weisman-T}, we have that $$d_{\mathbb{P}}\big(g^p y, g^p y')\leq \frac{\ell_2(g^p)}{\ell_1(g^p)}\frac{d_{\mathbb{P}}(y,y')}{(\min|y_i|)( \min|y_i'|)}\leq \frac{\zeta}{\epsilon^2}\frac{\ell_2(g^p)}{\ell_1(g^p)}.\qedhere$$\end{proof}

\begin{claim}\label{claim-ball2} Fix $x' \in \mathbb{P}(k^n)\smallsetminus \mathcal{N}_{\epsilon}(\mathbb{P}(V))$ and $0<\theta<\frac{\epsilon}{2}$. For any $w\in \mathcal{B}_{M,\epsilon}(x,V)$ and $s\in \mathbb{N}$, $$w^s B_{\theta}(x')\subset B_{\theta \sqrt[]{\frac{\ell_2(w^s)}{\ell_1(w^s)}}}(w^sx') \cap B_{\epsilon^{11}}([he_1]).$$\end{claim}

\begin{proof} Since $\mathbb{P}(he_1^{\perp})\subset \mathcal{N}_{\frac{\epsilon}{10}}(\mathbb{P}(V))$, we have  $\mathbb{P}(k^n)\smallsetminus \mathcal{N}_{\epsilon}(\mathbb{P}(V))\subset \mathbb{P}(k^n)\smallsetminus \mathcal{N}_{\frac{\epsilon}{2}}(\mathbb{P}(he_1^{\perp}))$. For any $x''\in B_{\theta}(x')$, write $x'=[v']$ and $x''=[v'']$ for some $v',v''\in k^n$ unit vectors and note that  $\big|(h^{-1}v'') \cdot e_1\big|\geq \frac{\epsilon}{10M^2}||h^{-1}v''||$, $\big|(h^{-1}v') \cdot e_1\big|\geq \frac{\epsilon}{10M^2}||h^{-1}v'||$. By using \cite[Lem. 2.1]{Weisman-T} we obtain \begin{align*}d_{\mathbb{P}}\big(w^s x'', w^s x'\big)&\leq C(h)\frac{\ell_2(w^s)}{\ell_1(w^s)}\frac{d_{\mathbb{P}}(h^{-1}x',h^{-1}x'')||h^{-1}v''||\cdot ||h^{-1}v'||}{|(h^{-1}v'')\cdot e_1|\cdot |(h^{-1}v')\cdot e_1|}\\ &\leq \frac{(10M)^6}{\epsilon^2}\frac{\ell_2(w^s)}{\ell_1(w^s)}d_{\mathbb{P}}(x',x'')\leq \sqrt[]{\frac{\ell_2(w^s)}{\ell_1(w^s)}}  \theta,\\ d_{\mathbb{P}}\big(w^s x'', [he_1]\big)&\leq C(h)\frac{\ell_2(w^s)}{\ell_1(w^s)}\frac{||h^{-1}v''||}{\big|( h^{-1}v'')\cdot e_1\big|}\leq \frac{(10M)^3}{L\epsilon}\leq \epsilon^{11}.\qedhere \end{align*}\end{proof}

Now we prove the estimate of the lemma. Consider a reduced product of the form $$g_0:=w_{r}^{s_r}g_{r}^{p_{r}}\cdots w_1^{s_1}g_1^{p_1}=\beta'a^{+}\beta,$$  written in the standard Cartan decomposition of $\mathsf{GL}_n(\mathbb{C})$, where $\beta,\beta'\in \mathsf{U}(n)$, $a^{+}\in \textup{Diag}_n(\mathbb{C})$ and $s_i\in \mathbb{N}\cup\{0\}$, $p_i\in \mathbb{Z}$, $g_i\in \mathcal{A}$, $w_i\in \mathcal{B}_{M,\epsilon}(x,V)$. There are two cases to consider.
\medskip

\noindent {\em Case 1: $p_1\neq 0$.} By \cite[Lem. 2.7]{Weisman-T} there is $y\in B_{\frac{\epsilon}{2}}(x)$ with $\textup{dist}(y, \beta^{-1}e_1^{\perp})\geq \frac{\epsilon}{4}$. By Claim \ref{claim-ball1}, for any $0<\zeta<\frac{\epsilon^4}{10^2}$,  $$g_1^{p_1} B_{\zeta}(y)\subset B_{\frac{\zeta \ell_2(g_1^{p_1})}{\epsilon^2 \ell_1(g_1^{p_1})}}(g^{p_1} y).$$ As $g_1^{p_1} y\in \mathbb{P}(k^n)\smallsetminus \mathcal{N}_{\epsilon}(\mathbb{P}(V))$ and $\frac{\zeta}{\epsilon^2}\frac{\ell_2(g_1^{p_1})}{\ell_1(g_1^{p_1})}<\epsilon^2$, by Claim \ref{claim-ball2}, $$w_1^{s_1}g_1^{p_1}B_{\zeta}(y)\subset B_{\mu}(w_1^{s_1}g_1^{p_1} y), \  \mu=\frac{\zeta}{\epsilon^2}\cdot \sqrt[]{\frac{\ell_2(w_1^{s_1})}{\ell_1(w_1^{s_1})}}\frac{\ell_2(g_1^{p_1})}{\ell_1(g_1^{p_1})}\leq \frac{\epsilon^4}{10^2}$$ and also $d_{\mathbb{P}}(w_1^{s_1}g_1^{p_1}y,[he_1])\leq \frac{\epsilon}{10}$.
By working inductively we deduce that $\gamma B_{\zeta}(y)\subset B_{\mu'}(\gamma y)$, where $\mu':=\zeta\epsilon^{-2r}\prod_{i=1}^{r} \left(\sqrt{\frac{\ell_2(w_i^{s_i})}{\ell_1(w_i^{s_i})}} \frac{\ell_1(g_i^{p_i})}{\ell_2(g_i^{p_i})}\right)$. Thus, since $\textup{dist}(y, \beta^{-1}e_1^{\perp})\geq \frac{\epsilon}{4}$, by \cite[Lem. 2.2]{Weisman-T} it follows that $\frac{\sigma_1(g_0)}{\sigma_2(g_0)}\geq \frac{\zeta \epsilon}{16\mu'}$. Since for every $i$, $\frac{\ell_1(w_i^{s_i})}{\ell_2(w_i^{s_i})}\geq \epsilon^{-12}$, we deduce (\ref{inequality}). Similarly, the same bound follows when $\gamma$ starts with a \hbox{non-trivial power of an element of $\mathcal{A}$.}

\medskip

\noindent {\em Case 2: $p_1=0$ and $s_1\neq 0$.} Fix $y' \notin \mathcal{N}_{\epsilon}(\mathbb{P}(V))$ with $\textup{dist}(y', \beta^{-1}e_1^{\perp})\geq \frac{\epsilon}{4}$. By Claim \ref{claim-ball2}, for $0<\zeta<\frac{\epsilon^4}{10^2}$ we have $w_1^{s_1}B_{\zeta}(y')\subset B_{\zeta\mu''}(gy')\cap B_{\epsilon^{11}}([he_1])$, $\mu''=\sqrt[]{\frac{\ell_2(w_1^{s_1})}{\ell_1(w_1^{s_1})}}$. By applying Case 1 for the reduced word $w_{r}^{s_r}g_{r}^{p_{r}}\cdots w_2^{s_2}g_2^{p_2}$, acting on $w_1^{s_1}B_{\zeta}(y')$, \hbox{we deduce the estimate.}\end{proof}

\begin{corollary}\label{svg-cor} Let $x\in \mathbb{P}(k^n)$, $V\in \mathsf{Gr}_{n-1}(k^n)$, $\mathcal{A}, \mathcal{B}_{M,\epsilon}(x,V)\subset \mathsf{GL}_n(k)$ and $M,\epsilon>0$ be as in Proposition \ref{svg-est}. Fix $\gamma \in \mathsf{GL}_n(k)$ with $||\gamma-\textup{I}_n||\leq  \epsilon^2$, $g_1\in \mathcal{A}$ a $1$-biproximal matrix and $g_2\in \mathcal{B}_{\frac{M}{2},\epsilon^2}(x,V)$. The semigroup  $\langle \gamma g_1\gamma^{-1}, \gamma g_1^{-1}\gamma ^{-1}, g_2\rangle$ is $1$-Anosov and isomorphic to $\mathbb{Z}\ast \mathbb{Z}^{+}$. In addition, if $g_2^{-1}\in \mathcal{B}_{\frac{M}{2},\epsilon^2}(x,V)$, then the group $\langle \gamma g_1 \gamma^{-1}, \gamma g_1^{-1} \gamma^{-1}, g_2, g_2^{-1}\rangle$ is $1$-Anosov and free of rank $2$.\end{corollary}

\begin{proof} By the definition of the set $\mathcal{B}_{M,\epsilon}(x,V)$ in (\ref{B(x,V)}), as $||\gamma^{\pm 1}||\leq 2$, $||\gamma^{\pm 1}-\textup{I}_n||\leq 2\epsilon^2$, $\gamma^{-1}\mathcal{B}_{\frac{M}{2},\epsilon^2}(x,V)\gamma$ is contained in $\mathcal{B}_{M,\epsilon}(x,V)$. Since $g_1\in \mathcal{A}$ and $\gamma^{-1}g_2\gamma \in \mathcal{B}_{M,\epsilon}(x,V)$, by Proposition \ref{svg-est} and (\ref{inequality}), for any reduced word $g_1^{p_r}g_2^{s_r}\cdots g_1^{p_1}g_2^{s_1}$, $p_i\in \mathbb{Z}$, $s_i\in \mathbb{N}\cup\{0\}$, the estimate holds \begin{align}\label{svg-2}\frac{\sigma_1( \gamma^{-1}g_2^{s_r}\gamma g_{1}^{p_{r}}\cdots \gamma^{-1}g_2^{s_1}\gamma g_1^{p_1})}{\sigma_2( \gamma^{-1}g_2^{s_r}\gamma g_{1}^{p_{r}}\cdots \gamma^{-1}g_2^{s_1}\gamma g_1^{p_1})}\geq \epsilon^{5}\min\left\{\sqrt[3]{\frac{\ell_1(g_2)}{\ell_2(g_2)}} , \ \frac{\ell_1(g_1)}{\ell_2(g_1)}, \frac{\ell_{n-1}(g_1)}{\ell_n(g_1)}\right\}^{\sum(|p_i|+|s_i|)}.\end{align} This shows that the semigroup $\langle g_1, g_1^{-1}, \gamma^{-1}g_2\gamma \rangle$ is $1$-Anosov. In addition, if $g_2^{-1}\in \mathcal{B}_{\frac{M}{2}, \epsilon^2}(x,V)$, then (\ref{svg-2}) holds for any reduced word in $\{g_1^{\pm 1},\gamma^{-1}g_2^{\pm 1}\gamma\}$ and the subgroup $\langle \gamma g_1 \gamma^{-1}, g_2\rangle$ of $\mathsf{GL}_3(\mathbb{C})$ is $1$-Anosov. \end{proof}

\begin{proof}[Proof of Theorem \ref{Archimedean-sem}] Let $h:=\begin{pmatrix}[1] 1 & 0 & \frac{1}{2}\\ 1 & 1& 1\\ 1 & 1 & -1 \end{pmatrix}$. Note that $||h^{\pm 1}||\leq 2$ and consider the flags \begin{align*}(x_1,V_1)&:=\big([he_1], (h^{-t}e_3)^{\perp}\big)=\big([e_1+e_2+e_3],(e_2-e_3)^{\perp}\big)\\ (x_2,V_2)&:=\big([he_3], (h^{-t}e_1)^{\perp}\big)=\big([e_2+2e_2-2e_3],(4e_1-e_2+e_3)^{\perp}\big).\end{align*} Fix $0<\epsilon\leq 10^{-6}$, $\lambda\geq 10^2$, $\kappa\geq \epsilon^{-28}$ and consider the matrices $$A:=\begin{pmatrix} \lambda & &\\ & -\lambda &\\ & &-\lambda^{-2}\end{pmatrix}, \ B:=h\begin{pmatrix} \kappa & &\\ & 1 &\\ & &\kappa^{-1}\end{pmatrix}h^{-1}.$$ Define the subsets $\Omega_1,\Omega_2$ of $\mathsf{GL}_3(\mathbb{C})$ as follows: \begin{align*}\Omega_1&:=\Omega_1^{+}\cup \Omega_1^{-}\\
\Omega_1^{+}&:=\left\{ \gamma \begin{pmatrix}\lambda_1 & & \\ & \lambda_2 & \\ & & \lambda_3 \end{pmatrix}\gamma^{-1}:  \begin{matrix}[1.1] ||\gamma-\textup{I}_3||<\epsilon^6\\ |\lambda_1-\lambda|<\epsilon^6, \ |\lambda_2+\lambda|<\epsilon^6\\ |\lambda_3+\lambda^{-2}|<\epsilon^6, \  |\lambda_1|\geq |\lambda_2| \end{matrix}  \right\}\\ \Omega_1^{-}&:=\left\{ \gamma \begin{pmatrix}\lambda_1 & & \\ & \lambda_2 & \\ & & \lambda_3 \end{pmatrix}\gamma^{-1}:  \begin{matrix}[1.1]||\gamma-\textup{I}_3||<\epsilon^6\\ |\lambda_1-\lambda|<\epsilon^6, \ |\lambda_2+\lambda|<\epsilon^6\\ |\lambda_3+\lambda^{-2}|<\epsilon^6, \  |\lambda_1|\leq |\lambda_2| \end{matrix}  \right\}\\\Omega_2&:=\left\{ (\gamma h)\begin{pmatrix}\kappa_1 & & \\ & \kappa_2 & \\ & & \kappa_3 \end{pmatrix}(\gamma h)^{-1}:  \begin{matrix}[1.1] ||\gamma-\textup{I}_3||<\epsilon^6, \ |\kappa_1-\kappa|<\epsilon^6\\  |\kappa_2-1|<\epsilon^6, \  |\kappa_3-\kappa^{-1}|<\epsilon^6 \end{matrix}  \right\}.\end{align*} Note that $\Omega_1,\Omega_2$ contain entirely diagonalizable matrices with distinct eigenvalues and define open neighborhoods of $A$ and $B$ in $\mathsf{GL}_3(\mathbb{C})$ respectively.

For any $\gamma \in \mathsf{GL}_3(\mathbb{C})$ with $||\gamma-\textup{I}_3||\leq \epsilon^6$, $||(\gamma h)^{\pm1}||\leq 3$. By the choice of $\kappa>1$ and $\epsilon>0$,  $\Omega_2$ is contained in $\mathcal{B}_{3, \epsilon^2}(x_1,V_2)$ and $\Omega_2^{-1}=\{g^{-1}:g\in \Omega_2\}$ is contained in $\mathcal{B}_{3,\epsilon^2}(x_2,V_1)$ (see (\ref{B(x,V)}) for the definition of the later sets). In addition, for any $A_{\pm}=\gamma \textup{diag}(\lambda_{1\pm},\lambda_{2\pm},\lambda_{3\pm})\gamma^{-1}$, where $||\gamma-\textup{I}_n||\leq \epsilon^6$ and $A_{\pm} \in \Omega_{1}^{\pm}$, and any $p\in \mathbb{N}$ we have the estimates:
\begin{align*} \frac{\big| (A_{+}^p he_1) \cdot h^{-t}e_1 \big|}{||A_{+}^p||}&\geq \frac{4|\lambda_{1+}^p|-|\lambda_{2+}^p|-|\lambda_{3+}^p|-10^4\epsilon |\lambda_{1+}^p|}{9|\lambda_{1+}^p|}\geq 10^{-2} \ \textup{since} \ |\lambda_{1+}|\geq |\lambda_{2+}|,\\
\frac{\big| (A_{+}^{-p} he_1) \cdot h^{-t}e_1 \big|}{||A_{+}^{-p}||}&\geq \frac{|\lambda_{3+}^{-p}|-4|\lambda_{1+}^{-p}|-|\lambda_{2+}^{-p}|-10^4\epsilon |\lambda_{3+}^{-p}|}{9|\lambda_{3+}^{-p}|}\geq 10^{-2},\\
 \frac{\big| (A_{-}^p he_3) \cdot h^{-t}e_3 \big|}{||A_{-}^p||}&\geq \frac{2|\lambda_{2-}^{p}|-2|\lambda_{3-}^{p}|-10^4\epsilon |\lambda_{2-}^{p}|}{9|\lambda_{2-}^{p}|}\geq 10^{-2} \ \textup{since} \ |\lambda_{2-}|\geq |\lambda_{1-}|,\\  
\frac{\big| (A_{-}^{-p} he_3) \cdot h^{-t}e_3 \big|}{||A_{-}^{-p}||}&\geq \frac{2|\lambda_{3-}^{-p}|-2|\lambda_{2-}^{-p}|-10^4\epsilon |\lambda_{3-}^{-p}|}{9|\lambda_{3-}^{-p}|}\geq 10^{-2}. \end{align*}

Thus, for any $p\in \mathbb{N}$ and $A_{\pm}\in \Omega_1^{\pm}$ we have the inclusions \begin{royalign}{0cm}\label{A'-item1}  A_{+}^{\pm p} B_{\epsilon}(x_1)&\subset \mathbb{P}(\mathbb{C}^3)\smallsetminus \mathcal{N}_{\epsilon}(\mathbb{P}(V_2)) \ \  &   \big|\big|A_{+}^{\pm p} \omega\big|\big|&\geq \sqrt[]{\lambda}^{p}||\omega|| & &\forall [\omega]\in B_{\epsilon}(x_1)\ \ \ \\ \label{A'-item2} A_{-}^{\pm p} B_{\epsilon}(x_2)&\subset \mathbb{P}(\mathbb{C}^3)\smallsetminus \mathcal{N}_{\epsilon}(\mathbb{P}(V_1)) \ \ & \big|\big|A_{-}^{\pm p} \omega\big|\big|&\geq \sqrt[]{\lambda}^{p}||\omega|| &\ \  &\forall [\omega]\in B_{\epsilon}(x_2).  \end{royalign}

For $\gamma \in \mathsf{GL}_3(\mathbb{C})$ with $||\gamma-\textup{I}_6||\leq \epsilon^6$, $$(h\gamma)\mathcal{N}_{\epsilon^3}(\mathbb{P}(e_3^{\perp}))\subset \mathcal{N}_{\epsilon}(\mathbb{P}(V_1)),\ (h\gamma)\mathcal{N}_{\epsilon^3}(\mathbb{P}(e_1^{\perp}))\subset \mathcal{N}_{\epsilon}(\mathbb{P}(V_2)),$$ hence for any $s\in \mathbb{N}$ and $B'\in \Omega_2$, \begin{royalign}{0cm}\label{B'-item1} (B')^s \big(\mathbb{P}(\mathbb{C}^3)&\smallsetminus \mathcal{N}_{\epsilon}(\mathbb{P}(V_2))\big)\subset B_{\epsilon}(x_1) & \big|\big|(B')^s \omega\big|\big|&\geq \sqrt[]{\kappa}^{s}||\omega|| & &\forall [\omega]\notin \mathcal{N}_{\epsilon}(\mathbb{P}(V_2)) \ \ \ \  \\ \label{B'-item2} (B')^{-s} \big(\mathbb{P}(\mathbb{C}^3)&\smallsetminus \mathcal{N}_{\epsilon}(\mathbb{P}(V_1))\big)\subset B_{\epsilon}(x_2) \ &   \big|\big|(B')^{-s} \omega\big|\big|&\geq \sqrt[]{\kappa}^{s}||\omega|| & \ \ &\forall [\omega]\notin \mathcal{N}_{\epsilon}(\mathbb{P}(V_1). \ \ \ \ \  \end{royalign}

By Lemma \ref{pingpong-lem}, for any $A'\in \Omega_1$ and $B'\in \Omega_2$, the semigroup $\langle A',(A')^{-1},B'\rangle$ is isomorphic to $\mathbb{Z}\ast \mathbb{Z}^{+}$ and quasi-isometrically embedded in $\mathsf{GL}_3(\mathbb{C})$, hence (\ref{Arch-sem2}) holds. In addition, as $A$ is not $1$-proximal and $A^{-1}$ is not $2$-proximal, the semigroup $\langle A,A^{-1},B\rangle$ is not Anosov in $\mathsf{GL}_3(\mathbb{C})$ and (\ref{Arch-sem1}) holds.
 Now fix $A'\in \Omega_1$ a $1$-biproximal matrix and $B'\in \Omega_2$. If $A'\in \Omega_1^{+}$, since $B'\in \mathcal{B}_{3, \epsilon^2}(x_1,V_2)$, by (\ref{A'-item1}),  (\ref{B'-item1}) and Corollary \ref{svg-cor}, it follows that the semigroup $\langle A',(A')^{-1},B'\rangle$ is $1$-Anosov. If $A'\in \Omega_1^{-}$, since $(B')^{-1}\in \mathcal{B}_{3, \epsilon^2}(x_2,V_1)$, by (\ref{A'-item2}), (\ref{B'-item2}) and Corollary \ref{svg-cor}, the semigroup $\langle A',(A')^{-1},(B')^{-1}\rangle$ is $1$-Anosov. This shows that the semigroup $\langle A',(A')^{-1},B'\rangle$ is $2$-Anosov and (\ref{Arch-sem3}) follows.\end{proof}

\begin{corollary}\label{lim-div} There exists a sequence of $1$-Anosov representations \hbox{$\{\rho_n:\mathsf{F}_2\rightarrow \mathsf{SL}_3(k)\}_{n\in \mathbb{N}}$} converging to a quasi-isometric embedding $\rho:\mathsf{F}_2\rightarrow \mathsf{SL}_3(k)$ whose image contains non-proximal matrices.\end{corollary}

\begin{proof} Let $\mathcal{A}:=\big\{\textup{diag}(\lambda_1, \lambda_2, \frac{1}{\lambda_1\lambda_2}): |\lambda_1|\geq |\lambda_2|\geq 2\big\}$ and \hbox{$(x,V):=\big([e_1+e_2+e_3], (2e_1-e_2-e_3)^{\perp}\big)$.} A computation shows that for any $g\in \mathcal{A}$ and $p\in \mathbb{Z}^{\ast}$, $\textup{dist}\big(g^p x,\mathbb{P}(V))\geq \frac{1}{2\sqrt{18}}$ and $g^pB_{10^{-4}}(x)\subset \mathbb{P}(k^3)\smallsetminus \mathcal{N}_{10^{-4}}(\mathbb{P}(V))$. Since $x\in \mathbb{P}(V)$, choose  $\beta\in \mathsf{GL}_3(k)$ such that $[\beta e_1],[\beta e_3]\in B_{10^{-6}}(x)$ and $\mathbb{P}(\beta e_1^{\perp})\cup \mathbb{P}(\beta e_3^{\perp})\subset \mathcal{N}_{10^{-5}}(\mathbb{P}(V))$. Thus, there is $r>1$, such that any non-trivial power of the matrix $\beta \textup{diag}(10^{r}, 1, 10^{-r})\beta^{-1}$ is in $ \mathcal{B}_{M, 10^{-8}}(x,V)$, $M=\max\{||\beta||,||\beta^{-1}||\}$. By applying Corollary \ref{svg-cor} for $\epsilon:=10^{-4}$, for any $g\in \mathcal{A}$, the group $\langle g,\beta \textup{diag}(10^{r}, 1, 10^{-r})\beta^{-1} \rangle$ is free and $1$-Anosov if and only if $g$ is $1$-biproximal.
It follows that for any $|\lambda|\geq 2$ the group $\Gamma_{\lambda}:=\big\langle \textup{diag}(\lambda, \lambda, \lambda^{-2}), \beta \textup{diag}(10^{r}, 1, 10^{-r})\beta^{-1}\big \rangle$ is free of rank $2$ and a limit of Anosov representations of $\mathsf{F}_2$ into $\mathsf{GL}_3(k)$. \end{proof}

\begin{rmk}\normalfont{ After sharing the above argument with Sami Douba, he pointed out to the author that, for $k= \mathbb{R}$, one can also deduce that the representations $\rho_n$ in the proof of Corollary \ref{lim-div} are Anosov using the combination theorem of Dey--Kapovich \cite{DK} for free products.}\end{rmk}

\section{Two examples in dimension 3}\label{Examples}
As an application of the proof of the Theorem \ref{nonArchimedean} and \ref{Archimedean}, we provide two examples of robust quasi-isometric embeddings with explicit open neighborhods which fail to contain Anosov representations into $\mathsf{SL}_3(k)$. Recall that for a matrix $h\in \mathsf{GL}_3(k)$, $C(h)=2||h||\cdot ||h^{-1}||$.

 \begin{Example}\label{Exmp1}  \normalfont{Let $k$ be a nonarchimedean local field and fix $\nu \in k$ with $|\nu|>{10^2}$. Define the representation $\rho:\mathsf{F}_2\rightarrow \mathsf{SL}_3(k)$, \begingroup
\addtolength{\jot}{0.5em} 
\begin{royalign*}{0.2cm}
\ \ \  \ \ \ \ \ \rho(a_1)&=\begin{pmatrix}
\nu  &  &   \\
& 1+\nu &   \\
 &  &\frac{1}{\nu+\nu^2}\end{pmatrix}   & \gamma&= b\begin{pmatrix} 1 & 0 & \nu^{7}\\  & 1 & 0\\  &  & 1
\end{pmatrix}b^{-1}  \\    
\rho(a_2)&=\gamma \begin{pmatrix} \nu^{150} &  & \\  & 1 & \\  &  & \nu^{-150}
\end{pmatrix}\gamma^{-1} & b&=\begin{pmatrix} 1 & 0 & 1\\ \nu & 1 & 1\\ -1-\nu & -1 & -1
\end{pmatrix}.
\end{royalign*}\endgroup The open neighborhood $\Omega$ of $\rho$, \begin{align}\label{open-exmp1}\Omega&:=\left\{\rho'\in \textup{Hom}\big(\mathsf{F}_2,\mathsf{GL}_3(k)\big): \begin{matrix}[1.2]  \big|\big|\rho'(a_1)-\rho(a_1)\big|\big|<|\nu|^{-31}\  \ \\ \big|\big|\rho'(a_2)-\rho(a_2)\big|\big|<|\nu|^{-190} \end{matrix} \right\}\end{align} consists entirely of quasi-isometric embeddings which are not Anosov.}\end{Example}

\begin{proof} Recall the notation in the proof of Theorem \ref{nonArchimedean} for $n=3$. Any non-trivial power of an element in the set $\Omega_{\frac{1}{|\nu|^{3}}}(\rho(a_1))$ satisfies Claim \ref{claim1} for $\epsilon:=|\nu|^{-3}$ and also note that $||b^{\pm 1}||\leq |\nu|,  C(b)\leq 2|\nu|^{2}, C(\gamma)\leq |\nu|^{19}.$  Since $\rho(a_2)^{+}=[\gamma e_1]$, $\rho(a_2)^{-}=[\gamma e_3]$, we have $$d_{\mathbb{P}}\big(\rho(a_2)^{\pm}, [be_1]\big)\leq 8C(b)\frac{||b^{-1}||}{|\nu|^{7}}<\frac{\epsilon}{5}, \ \textup{dist}(\mathbb{P}(V_{\rho(a_2)}^{\pm}),\mathbb{P}(be_3^{\perp}))\leq \frac{8C(b)||b||}{|\nu|^{6}}<\frac{\epsilon}{5},$$ hence $B_{\frac{\epsilon}{2}}(\rho(a_2)^{\pm})\subset B_{\epsilon}(y_3), \mathcal{N}_{\frac{\epsilon}{2}}(\mathbb{P}(V_{\rho(a_2)}^{\pm}))\subset \mathcal{N}_{\epsilon}(\mathbb{P}(V_3))$. In addition, $$\textup{dist}\big(\rho(a_2)^{\pm},\mathbb{P}(V_{\rho(a_2)}^{\mp})\big)\geq \min_{i=1,3}\big\{(||\gamma e_i||\cdot ||\gamma^{-t}e_{4-i}||)^{-1}\big\}\geq C(\gamma)^{-1}\geq |\nu|^{-19}.$$   The estimate of Lemma \ref{perturb1} for $\varepsilon:=|\nu|^{-19}$ and $\delta:=|\nu|^{-21}$, since $|\nu|^{150}>10^6 \frac{e^9C(\gamma)^4}{\varepsilon^{2}\delta}$, shows that for every $g'\in \big\{\rho(a_2)w\in \mathsf{GL}_3(k):||w^{\pm 1}-\textup{I}_3||<|\nu|^{-22}\big\}$ and $p\neq 0$, \begin{align*}(g')^p\mathcal{M}_{\varepsilon}(\rho(a_2))&\subset B_{\delta}(\rho(a_2)^{\pm})\\ (g')^p\big(\mathbb{P}(k^3)\smallsetminus \mathcal{N}_{\epsilon}(\mathbb{P}(be_3^{\perp}))\big)&\subset B_{\epsilon}([be_3])\end{align*} since $\mathbb{P}(k^3)\smallsetminus \mathcal{N}_{\epsilon}(\mathbb{P}(be_3^{\perp})) \subset \mathcal{M}_{\varepsilon}(\rho(a_2))$. As $||\rho(a_2)^{\pm 1}||\leq |\nu|^{168}$, it follows that $$\big\{g\in \mathsf{GL}_3(k):||\rho(a_2)-g||<|\nu|^{-190}\big\}\subset \big\{\rho(a_2)w\in \mathsf{GL}_3(k):||w^{\pm 1}-\textup{I}_3||<|\nu|^{-22}\big\}$$ and by Lemma \ref{perturb}, $\Omega_{\frac{1}{|\nu|^{3}}}(\rho(a_1))$ contains $\big\{g\in \mathsf{GL}_3(k): ||g-\rho(a_1)||<|\nu|^{-31}\big\}.$ Therefore, any representation of $\mathsf{F}_2$ in $\Omega$ in (\ref{open-exmp1}) is obtained as in the proof of Theorem \ref{nonArchimedean}.\end{proof}

\begin{Example}\label{Exmp2} \normalfont{Fix the presentation $$(\mathbb{Z}\times \mathsf{D}_8)\ast \mathbb{Z}=\big \langle z,t,s_1, s_2\ \big| \ s_1^4, s_2^2, (s_1s_2)^2, ts_1t^{-1}s_1^{-1}, ts_2t^{-1}s_2^{-1}  \big\rangle$$ and define the representation $\psi:(\mathbb{Z}\times \mathsf{D}_8)\ast \mathbb{Z}\rightarrow \mathsf{GL}_3\big(\mathbb{Q}\big)$,\begingroup
\addtolength{\jot}{0.5em}
 \begin{royalign*}{0.7cm}
\ \ \ \ \ \ \ \  \ \ \ \ \  \psi(s_1)&=\begin{pmatrix}
0 & 1 &   \\
- 1& 0 &   \\
 &  &1    \\
\end{pmatrix} & \psi(t)&=\begin{pmatrix}
10^5 & &    \\
 & 10^5&    \\
 &  & 10^{-10} 
\end{pmatrix}\\ \psi(s_2)&=\begin{pmatrix}
1 & 0 &    \\
 0& -1 &    \\
 &  &1   \\
\end{pmatrix} & \psi(z)&=\beta \begin{pmatrix} 10^{130} &  & \\  & 1 & \\  &  & 10^{-130}
\end{pmatrix}\beta^{-1}\\  h&=\begin{pmatrix} 1 & 0 & 1\\ 2 & 1 &1\\ 1 & -1 & 3
\end{pmatrix}  & \beta&= h\begin{pmatrix} 1 & 0 & 10^{7}\\  & 1 & 0\\  &  & 1\end{pmatrix}h^{-1} 
\end{royalign*} \endgroup  The open neighborhood $\Psi$ of $\psi$, \begin{align*}\Psi&:=\left\{\psi'\in \textup{Hom}\big((\mathbb{Z}\times \mathsf{D}_8)\ast \mathbb{Z},\mathsf{GL}_3(\mathbb{C})\big): \begin{matrix}[1.23]  \psi'(s_1)=w\psi(s_1)w^{-1}\\ 
\psi'(s_2)=w\psi(s_2)w^{-1}\\ 
 \end{matrix}\  \begin{matrix}[1.23] \ \ \ \ \ \ \ \ \big|\big|w-\textup{I}_3\big|\big|<10^{-12}\\
  \big|\big|\psi'(t)-\psi'(t)\big|\big|<10^{-12}  \\
 \big|\big|\psi'(z)-\psi'(z)\big|\big|<10^{-167}  \\ 
\end{matrix} \right\}   \end{align*} consists entirely of quasi-isometric embeddings which are not Anosov.}\end{Example}
\begin{proof} We are applying the proof of Theorem \ref{Archimedean} for the dihedral group of order 8, $\mathsf{D}_8= \langle s_1,s_2 \ | \ s_1^4,s_2^2, (s_1s_2)^2 \rangle$ and the irreducible representation $\psi_0:\mathsf{D}_8\rightarrow \mathsf{GL}_2(\mathbb{C})$, $$\psi_0(s_1)=\begin{pmatrix} 0 & 1 \\ -1 & 0\end{pmatrix}, \ \psi_0(s_2)=\begin{pmatrix} 1 &  0\\ 0 & -1\end{pmatrix}.$$ 

Recall the notation in the proof of Theorem \ref{Archimedean} for $d=2$, $n=3$, $\lambda=10^5$. Given $\psi_0$, we can take $v_1=\frac{e_1+2e_2}{\sqrt{5}}$ and $v_2=\frac{3e_1-e_2}{\sqrt{10}}$, $\eta=\frac{1}{10}$ so that (\ref{v1-ineq}) holds for $\psi_0$. Choose $\theta=10^{-3}$ and $\theta'=10^{-12}$ and let $\Psi_{10^{-12}}$ be the open neighborhood  of $\psi|_{\mathbb{Z}\times \mathsf{D}_8}$ defined in Lemma \ref{claim3}.
Note that $||h^{\pm 1}||\leq 4$, $C(h)\leq 2^5$, $||\beta^{\pm 1}||\leq 2^4 10^{7}$, $C(\beta)\leq 2^9 10^{14}$. A direct calculation shows that $d_{\mathbb{P}}\big(\psi(z)^{\pm}, [he_1]\big)<\frac{\theta}{2}$ and $\textup{dist}\big(\mathbb{P}(V_{\psi(z)}^{\pm}),\mathbb{P}(he_3^{\perp})\big)<\frac{\theta}{2}$, hence $B_{\frac{\theta}{2}}(\psi(z)^{\pm})\subset B_{\theta}(y)$, $\mathcal{N}_{\frac{\theta}{2}}(\mathbb{P}(V_{\psi(z)}^{\pm}))\subset \mathcal{N}_{\theta}(\mathbb{P}(he_3^{\perp}))$. Moreover, observe that $$\textup{dist}\big(\psi(z)^{\pm},\mathbb{P}(V_{\psi(z)}^{\mp})\big)\geq \min_{i=1,3}(||\beta e_i||\cdot ||\beta^{-t}e_{4-i}||)^{-1} \geq C(\beta)^{-1}.$$ By applying Lemma \ref{perturb1} for $\varepsilon:=\frac{1}{2^{9}10^{14}}$ and $\delta:=\frac{1}{2^{9}10^{16}}$, since $10^{130}>10^6 \frac{e^9C(\beta)^4}{\varepsilon^2 \delta}$ and $||\psi(z)^{\pm 1}||\leq 10^{147}$, for any matrix $g$ in $\big\{g\in \mathsf{GL}_3(\mathbb{C}):||\psi(z)-g||<10^{-167}\big\}$ and $p\neq 0$, $g^p\mathcal{M}_{\varepsilon}(g)\subset B_{\delta}(g^{\pm})$. In particular, since $\mathbb{P}(\mathbb{C}^3)\smallsetminus  \mathcal{N}_{\theta}(\mathbb{P}(he_3^{\perp}))\subset \mathcal{M}_{\varepsilon}(g)$, it follows that $g^p(\mathbb{P}(\mathbb{C}^3)\smallsetminus \mathcal{N}_{\theta}(\mathbb{P}(he_3^{\perp})))\subset B_{\theta}([he_1])$.  By Claim \ref{claim4}, for any $\psi'\in \Psi_{10^{-12}}$ and $c \in \mathbb{Z}\times \mathsf{D}_8$, non-trivial, $$\psi'(c)B_{\theta}([he_1])\subset \mathbb{P}(\mathbb{C}^3)\smallsetminus \mathcal{N}_{\theta}(\mathbb{P}(he_3^{\perp})),$$ thus any representation of $(\mathbb{Z}\times \mathsf{D}_8)\ast \mathbb{Z}$ in $\Psi_{10^{-12}}\times \big\{g\in \mathsf{GL}_3(\mathbb{C}):||\psi(z)-g||<10^{-167}\big\}$ is constructed as in the proof of Theorem \ref{Archimedean}.\end{proof}

\begin{appendix}\section{Technical Lemmas}
In this appendix, we obtain certain technical statements that we use throughout the paper.  Recall that $k$ is a local field and $|\cdot|$ is the absolute value on $k$. If $k$ is nonarchimedean, let $\textup{val}:k\rightarrow \mathbb{Z}\cup\{\infty\}$ be the discrete valuation on $k$ such that for any $x\in k$,  $|x|=e^{-\textup{val}(x)}$.  \begin{lemma}\label{dist} Let $h\in \mathsf{GL}_n(k)$ and set $C(h):=2||h||\cdot ||h^{-1}||$. The following estimates hold: \begin{enumerate}
\item \label{dist-item1}$d_{\mathbb{P}}(hx,x)\leq 2||h^{-1}|| \cdot ||h-\textup{I}_n||$ for every $x\in \mathbb{P}(k^n)$;
\item \label{dist-item2} $d_{\mathbb{P}}(hx,hy)\leq C(h)d_{\mathbb{P}}(x,y)$ for every $x,y\in \mathbb{P}(k^n)$.\end{enumerate}

\noindent In particular, for every $x\in \mathbb{P}(k^n)$, $W\in \mathsf{Gr}_{n-1}(k)$ and $\epsilon>0$, $$hB_{\epsilon}(x)\subset B_{\epsilon'}(x), \ h\mathcal{N}_{\epsilon}(\mathbb{P}(W))\subset \mathcal{N}_{\epsilon''}(\mathbb{P}(W))$$ where $\epsilon'=C(h)\epsilon+2||h^{-1}||\cdot ||h-\textup{I}_n||$ and $\epsilon''=C(h)\epsilon+2||h||\cdot ||h^{-1}-\textup{I}_n||$.
\end{lemma}

\begin{proof} The estimate in (\ref{dist-item1}) is the content of \cite[Lem. 2.2]{Tso-robust}.

\par For (\ref{dist-item2}), let $u,v\in k^n$ be unit vectors with $x=[u], y=[v]$. The triangle inequality implies \begin{align*}d_{\mathbb{P}}\big([hu],[hv]\big)&\leq \Bigg| \Bigg| \frac{hu}{||hu||}-\frac{hv}{ ||hv||}\Bigg|\Bigg|=\Bigg| \Bigg| \frac{h(u-v)}{||hu||}+\frac{||hu||-||hv||}{||hu||\cdot ||hv||}hv\Bigg|\Bigg|\\ &\leq ||h||\cdot ||h^{-1}||\cdot ||u-v||+\frac{1}{||hu||}\big|||hu||-||hv||\big|\\ &\leq \big(2||h||\cdot ||h^{-1}||\big)||u-v||.\qedhere  \end{align*} \end{proof}
\begin{proof}[Proof of Lemma \ref{Lip}] We prove the statement for the action of $g^p$ on $\mathcal{M}_{\theta}^{-}(g)$, $p\in \mathbb{N}$. The proof in the case of the action of $g^{-p}$ on $\mathcal{M}_{\theta}^{+}(g)$ is analogous.

Let $[\omega_1],[\omega_2]\in \mathcal{M}_{\theta}^{-}(g)$, $||\omega_1||=||\omega_2||=1$. By Lemma \ref{dist} (ii) we have $$C(h)\textup{dist}([h^{-1}\omega_i],e_{1}^{\perp})\geq \textup{dist}([\omega_i],he_{1}^{\perp})\geq \theta,$$ hence write $\frac{h^{-1}\omega_i}{||h^{-1}\omega_i||}=(x_i,v_i)\in k\oplus k^{n-1}$, $|x_i|\geq \theta  C(h)^{-1}$ for $i=1,2$. By letting $D_0:=\textup{diag}(\kappa_1,A_g,\kappa_n)$ and $D_1:=\kappa_1^{-1}\textup{diag}(A_g, \kappa_n)$, we obtain the estimates \begin{align*}d_{\mathbb{P}}\big([D_0^ph^{-1}\omega_1],[D_0^ph^{-1}\omega_2]\big)&\leq \Bigg| \Bigg| \frac{(1,x_1^{-1}D_1^pv_1)}{||(1,x_1^{-1}D_1^p v_1)||}- \frac{(1,x_2^{-1}D_1^pv_2)}{||(1,x_2^{-1}D_1^pv_2)||} \Bigg| \Bigg|\\ & \leq \frac{||D_1^p||\cdot \big|\big|x_1^{-1}v_1-x_2^{-1}v_2\big|\big|}{||(1,x_1^{-1}D_1^pv_1)||} +\frac{\big| ||(1,x_1^{-1}D_1^pv_1)||-||(1,x_2^{-1}D_1^pv_2)|| \big|}{||(1,x_1^{-1}D_1^pv_1)||}\\  & \leq 2||D_1||^p \big|\big|x_1^{-1}v_1-x_2^{-1}v_2\big|\big|\\ & \leq 2||D_1||^p \theta^{-2} C(h)^{2}\big(|x_1-x_2|+| |v_1-v_2| |\big)\\ & \leq 4 ||D_1||^p \theta^{-2} C(h)^2 d_{\mathbb{P}}\big([h^{-1}\omega_1],[h^{-1}\omega_2]\big)\\ & \leq 4 ||D_1||^p \theta^{-2} C(h)^3 d_{\mathbb{P}}\big([\omega_1],[\omega_2]\big). \end{align*} In particular, as $||D_1||\leq |\kappa_1^{-1}| \cdot ||A_g||$, we deduce that $$d_{\mathbb{P}}\big(hD_0^ph^{-1}[\omega_1],hD_0^ph^{-1}[\omega_2]\big)\leq \frac{4||A_g||^p}{ \theta^{2} |\kappa_1|^p} C(h)^4 d_{\mathbb{P}}\big([\omega_1],[\omega_2]\big).$$

Thus, the action of $g^p$ acting on $\mathcal{M}_{\theta}^{-}(g)$ is $4C(h)^4\theta^{-2}\Big(\frac{ ||A_g||}{|\kappa_1|}\Big)^p$-Lipschitz.
To see the bound from (\ref{norm-power}), for $[v]\in \mathcal{M}_{\theta}^{-}(g)$ write $\frac{h^{-1}v}{||h^{-1}v||}=(x_0,v')$, \hbox{$|x_0|\geq \theta C(h)^{-1}$, hence} \begin{align*}| | g^p v| |\geq \frac{||D_0^p(h^{-1}v)||}{| |h^{-1}| |}&\geq  |\kappa_1|^p |x_0|\frac{||h^{-1}v||}{| |h^{-1} | |}\geq  \frac{2\theta|\kappa_1|^p }{C(h)^2} | | v| |.\qedhere\end{align*} \end{proof}

\subsection{Perturbations of nonarchimedean matrices} In this subsection, we assume that $k$ is a nonarchimedean local field. For $A\in \mathsf{GL}_n(k)$ its norm is $||A||=\max_{i}|a_{ij}|$ and denote by $\chi_A(t)=\textup{det}(A-t\textup{I}_n)$ the characteristic polynomial of $A$. For two polynomials $f,g$, written in the form $f(t)=a_0+\cdots+a_nt^n$, $g(t)=b_0+\cdots+b_nt^n$, for some $n\geq 0$, define $||f-g||:=\max_j |a_j-b_j|$ and $\textup{val}(f-g):=\min_{j}\textup{val}(a_j-b_j)$. 

By using \cite[Thm. 32.20]{Warner} (see also \cite[Thm. 36]{continuity-roots}) and the estimate provided therein, we obtain the following lemma providing explicit open neighborhoods of diagonalizable matrices with distinct eigenvalues.

\begin{lemma}\label{perturb} Let $C=\textup{diag}(a_1,\ldots,a_n)\in \mathsf{SL}_n(k)$, $n\geq 2$, be a diagonal matrix with $a_i\neq a_j$ for $i\neq j$. Fix $\epsilon\in (0,1)$ with $0<\epsilon<\frac{1}{2} \underset{i\neq j}{\min}|a_i-a_j|$ and let  \begin{align}\label{constant-0}\theta(n,\epsilon, C):=\min \Big\{10^{-2}\underset{i\neq j}{\min}|a_i-a_j|, \epsilon^n ||C||^{1-2n^2-n}\Big\}.\end{align} For every $C'\in \mathsf{GL}_n(k)$ with $||C-C'||<\theta(n,\epsilon, C)$, there are $w\in \mathsf{GL}_n(k)$, $a_1',\ldots,a_n'\in k$ with $$C'= w\textup{diag}\big( a_1', \ldots, a_n'\big)w^{-1},$$  $||w^{\pm 1}-\textup{I}_n| |<\epsilon $  and $\max_{j}|a_j'-a_j|<\epsilon$.\end{lemma}

\begin{proof} Write $C=(c_{ij})_{i,j=1}^{n}$, $C'=(c_{ij}')_{i,j=1}^{n}$ and $\chi_{C}(t)=\prod_i (t-a_i)=\sum_{i}A_it^i$.  \hbox{Since $||C||\geq 1$,} $||C'||\geq \frac{1}{2}$ and $||C-C'||<\theta(n,\epsilon,C)$, by the ultrametric property, it follows that $||C'||=||C'||$. For any collection of pairs $\{(i_1,j_1),\ldots, (i_r,j_r)\}$, $r\leq n$,  write $$c_{i_1j_1}\cdots c_{i_rj_r}-c_{i_1j_1}'\cdots c_{i_rj_r}'=(c_{i_1j_1}-c_{i_1j_1}')c_{i_2j_2}\cdots c_{i_rj_r}+\cdots +c_{i_1j_1}'\cdots c_{i_{r-1}j_{r-1}}'(c_{i_rj_r}-c_{i_rj_r}')$$ hence, by using the triangle inequality and the ultrametric property, it follows that $$\max_{1\leq r\leq n} \big|c_{i_1j_1}\cdots c_{i_rj_r}-c_{i_1j_1}'\cdots c_{i_rj_r}'\big|\leq ||C-C'|| \cdot ||C||^{n-1}$$ and $|c_{i_1j_1}\cdots c_{i_rj_r}|\leq ||C||^r$. Therefore, we obtain the estimates \begin{align}\label{char-poly}||\chi_{C}-\chi_{C'}|| \leq ||C-C'||\cdot ||C||^{n-1},\end{align} $||\chi_{C}||=\max_{j}|A_j|\leq ||C||^n$, $-\min_{j}\textup{val}(A_j)\leq \log||C||^n$. Since $||C-C'||\leq \theta(n,\epsilon,C)$, by (\ref{char-poly}), \begin{align*}\label{diagonal-ineq}\textup{val}\big(\chi_{C}-\chi_{C'}\big)&> n\left(-\log \epsilon+2\log ||C||^n\right)\geq  n \big(-\log \epsilon-\min_{0\leq j\leq n}\textup{val}(A_j)+\max\{-\min_{0\leq j\leq n}\textup{val}(A_j),0\}\big).\end{align*} Thus, as $t-a_1,\ldots,t-a_n\in k[t]$ are distinct, irreducible and  \hbox{$-\log \epsilon>\max \big\{0, \underset{i\neq j}{\max}\ \textup{val}(a_i-a_j)\big\}$,} by \cite[Thm. 36 \& Thm. 16]{continuity-roots} for $\varepsilon:=-\log \epsilon$ (and the estimate provided therein), for any $j\in \{1,\ldots,n\}$ there is a monic polynomial $t-a_j'\in k[t]$ such that $\textup{val}(a_j-a_{j}')>-\log\epsilon$ and $\chi_{C'}(t)=\prod_{j=1}^{n}(t-a_j')$. Since for every $j$, $|a_j-a_j'|\leq \epsilon \leq \frac{1}{10}\underset{i\neq j}{\min}|a_i-a_j|$, $\chi_{C'}$ has all of its roots distinct. In particular, $C'$ admits $n$ linearly independent eigenvectors and there is $h\in \mathsf{GL}_n(k)$ with $||h||=1$, $$C'=hDh^{-1}, \ D:=\textup{diag}\big(a_{1}',\ldots,a_{n}'\big).$$ Let us write $h=D_{h}+X_{h}$, where $D_{h}$ is the diagonal matrix formed by the diagonal entries of $h$ and $X_h$ has its diagonal entries equal to zero. Let $D_{h}'$ be the $n\times n$ diagonal matrix whose $(i,i)$ entry is equal to $0$ if $(D_{h})_{ii}\neq 0$, otherwise take $(D_{h}')_{ii}$ to be non-zero and of modulus less or equal to $\theta(n, \epsilon,C)$. By the choice of $D_{h}'$, note that $D_{h}+D_{h}'$ is invertible and set $h':=h+D_{h}'=(D_{h}+D_{h}')+X_{h}$, so that $\frac{1}{2}\leq ||h'||\leq 2$.

By using the triangle inequality we have that \begin{align*} | | C h'- h'D | |&=| | C h- hD+(C-D)D_h'| |\leq | |C-C'| |\cdot || h||+| |D_{h}'| |\cdot | |C-D| |\leq 4\theta(n, \epsilon,C), \end{align*} hence, as $C$ and $D$ are diagonal, we deduce \begin{align*}| | CX_{h}-X_{h}D| |\leq 4\theta(n,\epsilon,C).\end{align*} This shows $|(X_{h})_{ij}(a_i-a_j)|\leq 4\theta(n, \epsilon,C)$ for $i\neq j$. As $X_{h}$ has all of its diagonal entries equal to $0$, it follows that $$||X_{h}||\leq \frac{4\theta(n, \epsilon,C)}{\underset{i\neq j}{\min}|a_i-a_j|}\leq \frac{4\epsilon^2}{\underset{i\neq j}{\min}|a_i-a_j|}\leq \frac{\epsilon}{25}.$$ Since $\frac{1}{2}\leq ||h'||\leq 2$ and $h'=(D_{h}+D_{h}')+X_{h}$ we also have $\frac{1}{4}\leq \big|\big|(D_{h}+D_{h}')^{\pm 1}\big|\big|\leq 4$. 

\par Finally, set $w:=\textup{I}_n+X_{h}(D_{h}+D_{h}')^{-1}=h'(D_{h}+D_{h}')^{-1}$ and note $C'=hDh^{-1}=h'D(h')^{-1}=wDw^{-1}$ and  $||w-\textup{I}_n||\leq \frac{\epsilon}{4}$. In particular, $||w^{-1}||\leq 1+\frac{\epsilon}{2}||w^{-1}||$, hence $||w^{-1}||\leq 2$ and $||w^{-1}-\textup{I}_n||\leq ||w^{-1}||\cdot ||w-\textup{I}_n||< \epsilon$. This concludes the proof of the lemma.\end{proof}

\end{appendix}

\bibliographystyle{siam}

\bibliography{biblio.bib}

\end{document}